\documentclass[10pt]{amsart}
\usepackage{amssymb, accents, tikz, tikz-cd, tabto, stmaryrd}

\theoremstyle{definition}
\newtheorem{theorem}{Theorem}[section]

\newtheorem{remark}[theorem]{Remark}
\newtheorem{definition}[theorem]{Definition}
\newtheorem{example}[theorem]{Example}

\newtheorem{proposition}[theorem]{Proposition}

\numberwithin{equation}{section}

\allowdisplaybreaks

\newcommand{\N}{\mathbb{N}}
\newcommand{\Z}{\mathbb{Z}}
\newcommand{\del}{\partial}
\newcommand{\Del}{\Delta}
\newcommand{\unC}{\underline{C}}
\newcommand{\undel}{\underline{\partial}}
\newcommand{\unDel}{\underline{\Delta}}
\newcommand{\Ss}{\mathcal{S}}
\newcommand{\ot}{\otimes}
\newcommand{\bu}{\bullet}
\newcommand{\ds}{\scalebox{0.5}[1]{$\dots$}}
\newcommand{\wt}{\widetilde}
\newcommand{\wh}{\widehat}
\newcommand{\bcr}{\scalebox{1.5}{$\bullet$}}
\newcommand{\bi}{{i}}
\newcommand{\bj}{{j}}
\newcommand{\bo}[1]{{#1}}

\newcommand{\ub}{\underbrace}
\newcommand{\lb}[2]{\text{(\ref{#1}#2)}}
\newcommand{\un}[1]{\underline{#1}}

\newcommand{\rA}[2]{\big({\bo{#1}}\big) \bo{#2}}
\newcommand{\rB}[4]{\big({\bo{#1}}\big| {\bo{#2}}\big| {\bo{#3}}\big) \bo{#4}}
\newcommand{\rC}[6]{\big({\bo{#1}}\big| {\bo{#2}}\big| {\bo{#3}}\big| {\bo{#4}}\big| {\bo{#5}}\big) \bo{#6}}
\newcommand{\rD}[8]{\big({\bo{#1}}\big| {\bo{#2}}\big| {\bo{#3}}\big| {\bo{#4}}\big| {\bo{#5}}\big| {\bo{#6}}\big|{\bo{#7}}\big) \bo{#8}}

\title{An algebraic model for the constant loops map}

\author[L.~Fernandez]{Luis~Fernandez}
  \address{Luis Fernandez,
  Department of Mathematics, Bronx Community College, City University of New York, 2155 University Avenue, Bronx, New York 10453}
  \email{luis.fernandez01@bcc.cuny.edu}

  \author[M.~Rivera]{Manuel~Rivera}
  \address{Manuel Rivera,
  Department of Mathematics, Purdue University, 150 N University St, West Lafayette, Indiana, 47907 }
  \email{manuelr@purdue.edu}

\author[T.~Tradler]{Thomas~Tradler}
  \address{Thomas Tradler,
  Department of Mathematics, New York City College of Technology, City University of New York, 300 Jay Street, Brooklyn, NY 11201}
  \email{ttradler@citytech.cuny.edu}


\begin{document}
\maketitle
\begin{abstract}
For any simplicial complex $X$ with a total ordering of its vertices, one can construct a chain complex $\mathbb{L}_\bullet(X)$ generated by necklaces of simplices in $X$, which computes the homology of the free loop space of the geometric realization of $X$. Motivated by string topology, we describe two explicit chain maps $C_\bullet(X) \to \mathbb{L}_\bullet(X)$, where $C_\bullet(X)$ denotes the simplicial chains in $X$, lifting the homology map induced by embedding points in $|X|$ into constant loops in the free loop space of $|X|$. One of the maps has a convenient combinatorial description, while the other is described in terms of higher structure on $C_\bullet(X)$.

%For the chains $C$ of xa simplicial complex, we provide an explicit  chain map to the coHochschild complex $\cCH(C,C)$ of $C$. Applying this map to a fundamental cycle of a manifold, yields a non-degenerate homotopy inner product, i.e., a weak Calabi-Yau structure. This can be used, for example, as the starting point for defining string topology operations in an algebraic setting. We apply this to the case of the $2$-sphere, extending an earlier construction from $\Z_2$ to $\Z$ coefficients.
\end{abstract}

\section{Introduction}\label{introduction}

Let $X$ be a simplicial complex with a total ordering of its vertices, and denote by $(C_\bullet(X),\partial,\Delta)$ the differential graded (dg) coalgebra of simplicial chains on $X$ over a commutative ring with unit $R$. We consider a chain complex $(\mathbb{L}_\bullet(X), D)$, built directly from the combinatorics of $X$, that computes the homology of $L|X|= \text{Maps}(S^1,|X|)$, the free loop space of the geometric realization of $X$. 
The goal of this article is to give explicit chain level models of the %map on homology $H_\bullet(|X|) \to H_\bullet(L|X|)$ induced by the 
embedding $\iota:|X| \hookrightarrow L|X|$ that sends a point $x \in |X|$ to the constant loop at $x$. In fact, we describe \textit{two} chain homotopic natural maps $\rho$ and $\chi$,
\[ \rho, \chi: C_\bullet(X) \to \mathbb{L}_\bullet(X)\]
whose induced maps on homology $H_\bu(|X|)\to H_\bu(L|X|)$ coincide with $H_\bu(\iota)$. 

The chain map $\rho$ is defined by an explicit combinatorial formula making it particularly suitable for computations in string topology, especially in situations where constant loops play a central role.
The chain map $\chi$ is described in terms of the natural higher algebraic structure of $C_\bullet(X)$, giving a conceptual explanation of the algebraic ingredients necessary to construct such a lift.

We envision these constructions being useful in explicit computations of string topology operations such as the \textit{Goresky-Hingston coproduct}. This is a coproduct operation defined on the relative homology $H_\bullet(LM, M)$ of the free loop space of a closed smooth manifold $M$ modulo constant loops, which is sensitive to geometric information beyond the homotopy type of $M$ \cite{N, NRW}. The approach we envision proceeds as follows: 1) begin with a triangulation of $M$ yielding a simplicial complex $X$; 2) lift Poincaré duality to the chain level as an appropriate homotopical structure on $C_\bullet(X)$; 3) define the Goresky--Hingston coproduct as an operation on the cokernel of $\rho$ (or $\chi$). The resulting model is explicit and transparent enough to reveal how geometric features of the manifold $M$ are reflected in the structure of its free loop space. 

 Moreover, the maps $\rho$ and $\chi$ may be used to produce closed elements in $\mathbb{L}_\bullet(X)$, which, when coming from suitable top dimensional cycles in $C_\bu(X)$, can be used to construct explicit Calabi-Yau structures; see Remark \ref{REM:CY} and Example \ref{EXA:2-sphere}.

\subsection{}
To state our results in more detail, we recall the construction of $(\mathbb{L}_\bullet(X), D)$ from \cite{R}, which, informally, is obtained from the coHochschild complex of appropriate interpretation of $C_\bu(X)$ by formally inverting all $1$-simplices. Let $X_j$ be the set of $j$-simplices in $X$ and denote by $\mathsf{s},\mathsf{t} \colon \bigcup_{j=0}^{\infty}X_j \to X_0$ the functions sending a simplex to its first and last vertices, respectively. Consider a new set $X_1^{-1}= \{ \sigma^{-1} : \sigma \in X_1\}$, which will play the role of ``formal inverses" for the elements of $X_1$. Extend $\mathsf{s}$ and $\mathsf{t}$ to $X_1^{-1}$ by $\mathsf{s}(\sigma^{-1}) = \mathsf{t}(\sigma)$ and $\mathsf{t}(\sigma^{-1})=\mathsf{s}(\sigma)$. Given a commutative ring with unit $R$, define $\mathbb{L}_n(X)$ to be the free $R$-module generated by all ordered sequences $(\sigma_1 | \cdots | \sigma_p )\sigma_{p+1}$ such that:
\begin{enumerate}
\item $p$ is a non-negative integer, $\sigma_1, \ldots, \sigma_p \in \Big(\bigcup_{j=1}^{\infty} X_j\Big) \cup X_1^{-1}$, and $\sigma_{p+1} \in \bigcup_{j=0}^{\infty} X_j$,
\item $|\sigma_1| + \cdots + |\sigma_p|-p+ |\sigma_{p+1}| =n$, where $|\sigma|$ denotes the dimension of a simplex $\sigma$,
\item $\mathsf{t}(\sigma_i)=\mathsf{s}(\sigma_{i+1})$ for $i=1,\ldots,p$ and $\mathsf{t}(\sigma_{p+1})=\mathsf{s}(\sigma_1)$, and
\item $(\sigma_1 | \cdots | \sigma_p ) \sigma_{p+1}$ is reduced, i.e., a $1$-simplex never appears next to its formal inverse. 
\end{enumerate}
A sequence $(\sigma_1 | \cdots | \sigma_p )\sigma_{p+1}$ satisfying (1)-(4) is called a \textit{necklace of dimension $n$} in $X$ with \textit{beads} $\sigma_1, \ldots, \sigma_{p}$ and \textit{marked bead} $\sigma_{p+1}$. Necklaces with $p=0$ have a marked bead $\sigma$ and no other beads; in this case, $\sigma$ must be a vertex in $X_0$ and we write the necklace as $(\text{id}_{\sigma})\sigma \in \mathbb{L}_0(X)$.  

The differential $D \colon \mathbb{L}_n(X) \to \mathbb{L}_{n-1}(X)$ is defined by sending a given generator $(\sigma_1|\cdots| \sigma_p) \sigma_{p+1} \in \mathbb{L}_n(X)$ to the appropriately signed sum of all the generators $(\alpha_1 | \cdots |\alpha_q) \alpha_{q+1} \in \mathbb{L}_{n-1}(X)$ such that each of $\alpha_1, \ldots, \alpha_q$ is a simplex in the simplicial complex $\sigma_1 \cup \cdots \cup \sigma_{p+1}$ and $\alpha_{q+1}$ is a sub-simplex of $\sigma_{p+1}$. Explicitly, the differential $D$ is essentially given by applying the (reduced) boundary and coproduct of $C_\bu(X)$ to each bead as well as the marked bead (similarly to the differential in the coHochschild complex); see \eqref{EQU:(L,D)=(PxC,d+} and \eqref{necklaceD} below for more details. 

We recall from \cite[Theorem 20]{R}  that for any ordered simplicial complex $X$, the chain complex $(\mathbb{L}_\bullet(X),D)$ is quasi-isomorphic to $(S_\bullet(L|X|), \partial)$ the singular chains on the free loop space of the geometric realization of $X$.

\subsection{}
Our first map $\rho: C_\bullet(X) \to \mathbb{L}_\bullet(X)$ describing a model of the inclusion of constant loops is given by an explicit formula. To state it, we use the notation that for any $\sigma \in X_k$ with vertices $\bi_0< \dots <\bi_k$ we write $\sigma= [\bi_0,\dots,\bi_k]\in C_k(X)$. The formal inverse of a $1$-simplex $[a,b]$ will be denoted by $[b,a]$. Then, we define $\rho:C_\bullet(X) \to \mathbb{L}_\bullet(X)$ by setting $\rho([\bi_0])=(\text{id}_{[\bi_0]})[\bi_0]$, and, for $\sigma= [\bi_0,\dots,\bi_k]$ with $k>0$, define $\rho([\bi_0,\dots,\bi_k]) \in \mathbb{L}_k(X)$ to be
\begin{align}\label{EQU:rho} 
& \rho([\bi_0,\dots,\bi_k])=\sum_{\ell\geq 1}\, \,\sum_{s=((p_1,q_1),\dots,(p_\ell,q_\ell))\in \Ss^{(\ell)}_k} \,\,(-1)^{\epsilon(s)} \cdot \sigma\llbracket s\rrbracket, \quad \text{where}
\\ \label{EQU:sigma[[s]]} & \sigma\llbracket s\rrbracket=
\Big( [\bi_{q_1},\bi_{p_1}] \Big| [\bi_{p_1},\dots,\bi_{p_2},\bi_{q_1},\dots,\bi_{q_2}] \Big| 
\\
\nonumber
& \hspace{13.8mm}[\bi_{q_2},\bi_{p_2}] \Big| [\bi_{p_2},\dots,\bi_{p_3},\bi_{q_2},\dots,\bi_{q_3}] \Big| \quad \dots \quad \Big| 
\\
\nonumber
& \hspace{13.8mm} 
 [\bi_{q_{\ell-1}},\bi_{p_{\ell-1}}] \Big| [\bi_{p_{\ell-1}},\dots,\bi_{p_\ell},\bi_{q_{\ell-1}},\dots,\bi_{q_\ell}]  \Big| 
 [\bi_{q_\ell},\bi_{p_\ell}] \Big)  [\bi_{p_\ell},\dots,\bi_{q_1}]
\end{align}
Here, for $k\geq 0$ and $\ell\geq 1$, the indexing set $\Ss^{({\ell})}_k$ is defined to be the set of all sequences of tuples $(p_j,q_j)\in \N_0\times\N_0$ of length $\ell$ subject to the following conditions:
\begin{align}\label{EQU:S^l_k}
\Ss^{({\ell})}_k=\{s=((p_1,q_1),\dots,(p_\ell,q_\ell))  :& \quad  
%p_1<q_1,\,\, p_2<q_2,\,\, \dots, \,\, p_\ell<q_\ell, \\ \nonumber
  p_1+q_1<p_2+q_2<\dots<p_\ell+q_\ell, \\ \nonumber
 &\hspace{-20mm} 0=p_1\leq p_2\leq \dots\leq p_\ell\leq q_1\leq q_2\leq \dots \leq q_\ell=k, \\ \nonumber
 & \hspace{-6mm}\text{and }\quad p_2<q_1,\,\,  p_3<q_2\,\, ,\dots, \,\, p_\ell<q_{\ell-1} \quad  \}
\end{align}
Note that the conditions for $s=((p_1,q_1),\dots,(p_\ell,q_\ell))\in \Ss^{(\ell)}_k$ in \eqref{EQU:S^l_k} also imply that $ p_1<q_1$, $p_2<q_2$, $\dots$, $p_\ell<q_\ell$. The sign $\epsilon(s)$ in \eqref{EQU:rho} for $s=((p_1,q_1),\dots,(p_\ell,q_\ell))\in \Ss^{(\ell)}_k$ is given by $\epsilon(s)=\ell+1 +p_\ell+q_1\cdot (q_\ell+1)+\sum_{j=2}^{\ell} p_j\cdot (q_j-q_{j-1})$.
\begin{theorem}\label{Theorem1}
The map $\rho \colon (C_\bullet(X),\partial) \to (\mathbb{L}_\bullet(X),D)$ defined above is a chain map. On homology, the map coincides with the map induced by the constant loops embedding $\iota:|X| \hookrightarrow L|X|$.
\end{theorem}
Explicit examples for formula \eqref{EQU:rho} are given in Example \ref{EXA:0123-simplicies}. The combinatorics involved in the indexing set may be better understood through a picture describing a recipe of how to produce a typical summand \eqref{EQU:rho}; see Remark \ref{RMK:rho}.

\subsection{}
Before we can define our second map $\chi$ modeling the inclusion of constant loops, we first need to recall how $\mathbb{L}_\bu(X)$ can be written as an appropriate product of a chain complex $\mathbb{P}_\bullet(X)$ with $C_\bu(X)$, which we describe now.

For the $R$-coalgebra $C_0=(C_0(X), \Delta)$ obtained by equipping the $R$-module $C_0(X)$ of linear combinations of vertices in $X$ with the coproduct determined by $\Delta(a)=a \otimes a$ on any $a \in X_0$, we consider the monoidal category of dg $C_0$-bicomodules with underlying free $R$-modules and cotensor product $\underset{C_0}{\square}$. Note, that if we modify the differential $\partial\colon C_\bullet(X) \to C_{\bullet-1}(X)$, $ {\partial}([\bi_0,\dots, \bi_k])=\sum_{j=0}^{k} (-1)^j\cdot [\bi_0, \dots, \widehat{\bi_j}, \dots, \bi_k]$, to a reduced differential $\widetilde{\partial} \colon C_\bullet(X) \to C_{\bullet-1}(X)$,
\begin{equation} \label{modifiedpartial}
\widetilde{\partial}([\bi_0,\dots, \bi_k])= \sum_{j=1}^{k-1} (-1)^j[\bi_0, \dots, \widehat{\bi_j}, \dots, \bi_k],
\end{equation}
then $(C_\bullet(X), \widetilde{\partial}, \Delta)$ together with graded $C_0$-bicomodule structure given by  the maps $\mathsf{s}$ and $\mathsf{t}$ defines a comonoid in the monoidal category of dg $C_0$-bicomodules.

Next, define $\mathbb{P}_\bullet(X)$ to be the graded $R$-module generated by the symbols $(\text{id}_{\sigma})$, for all $\sigma \in X_0$, and all reduced sequences $(\sigma_1 | \dots | \sigma_p)$ with $\sigma_1, \ldots, \sigma_p \in \Big(\bigcup_{j=1}^{\infty} X_j\Big) \cup X_1^{-1}$,  $\mathsf{t}(\sigma_i)=\mathsf{s}(\sigma_{i+1})$ for $i=1,\ldots,p-1$, and $|\sigma_1| + \cdots + |\sigma_p|-p=n$. Then, $\mathbb{P}_\bullet(X)$ becomes a graded $C_0$-bicomodule with structure maps induced by $(\sigma_1 | \dots | \sigma_p) \mapsto \mathsf{s}(\sigma_1)$ and $(\sigma_1 | \dots | \sigma_p) \mapsto \mathsf{t}(\sigma_p)$. Moreover, we equip $\mathbb{P}_\bullet(X)$ with a differential $d \colon \mathbb{P}_n(X) \to \mathbb{P}_{n-1}(X)$ that sends a generator $(\sigma_1 | \dots | \sigma_p) \in \mathbb{P}_n(X)$ to the appropriately signed sum of all generators $(\alpha_1 | \cdots | \alpha_q) \in \mathbb{P}_{n-1}(X)$ such that $\mathsf{s}(\alpha_1)=\mathsf{s}(\sigma_1)$, $\mathsf{t}(\alpha_q)=\mathsf{t}(\sigma_p)$, and each of $\alpha_1, \ldots, \alpha_q$ is a simplex of the simplicial complex $\sigma_1 \cup \cdots \cup \sigma_p$. Together with the map $\mu \colon \mathbb{P}_\bullet(X) \underset{C_0}{\square} \mathbb{P}_\bullet(X) \to \mathbb{P}_\bullet(X)$,  induced by the concatenation of sequences, $\mathbb{P}_\bullet(X)$ becomes a monoid in the category of dg $C_0$-bicomodules; this is a many-object version of Adams' cobar construction that models the dg category of paths in $|X|$ \cite{A,R,R2}.

Using the above, we see that there is a natural isomorphism of chain complexes
\begin{equation}\label{EQU:(L,D)=(PxC,d+}
\big(\mathbb{L}_\bullet(X),D\big) \cong \big(\mathbb{P}_\bullet(X) \underset{C_0 \otimes C_0^{\text{op}}}{\square} C_\bullet(X), d \otimes \text{id} + \text{id} \otimes \widetilde{\partial} + \delta \big),
\end{equation}
where $\delta$ is defined in terms of $\Delta$ together with the left- and right- $C_0$-module structures of $\mathbb{P}_\bullet(X)$; see \eqref{EQU:delta-in-P} for more details.

%In fact, as discussed in detail in the next section, the differential $d$ may be described by extending $\widetilde{\partial}+\widetilde{\Delta}$ as a derivation (with appropriate signs), where \[\widetilde{\Delta}([\bi_0, \ldots, \bi_k]) = \sum_{j=1}^{k-1} \big([\bi_0, \ldots, \bi_j] | [\bi_j, \ldots, \bi_k] \big).\]
%For each $a,b \in X_0$ denote by $(\mathbb{P}_\bullet(X)(a,b),d)$ the sub-complex of $(\mathbb{P}_\bullet(X),d)$ generated by all $\big(\sigma_1 | \cdots |\sigma_p\big)$ such that $\mathsf{s}(\sigma_1)=a$ and $\mathsf{t}(\sigma_p)=b$. 
We then consider the following additional structure:
\begin{itemize}
\item a degree $0$ coproduct \[\nabla_0 \colon \mathbb{P}_\bullet(X) \to \mathbb{P}_\bullet(X) \otimes \mathbb{P}_\bullet(X)\] extending $\nabla_0(\sigma) = \big(\sigma\big) \otimes \big(\sigma\big)$ for any $\sigma \in X_1$ and making $(\mathbb{P}_\bullet(X),d)$ a dg coassociative counital coalgebra (Section \ref{SEC:nabla_0}), and %and satisfies $\nabla_0([a,b]) = ([a,b]) \otimes ([a,b])$ when $[a,b] \in X_1$, and
\item a chain homotopy \[\nabla_1 \colon \mathbb{P_\bullet }(X)\to \mathbb{P}_\bullet(X)\otimes \mathbb{P}_\bullet (X)\]  between $\nabla_0$ and $\nabla_0^{\text{op}}$ (Section \ref{SEC:nabla_1}). 
\end{itemize}
%Denoting by $(\mathbb{P}_\bullet(X)(a,b),d)$ the sub-complex of $(\mathbb{P}_\bullet(X),d)$ generated by all $\big(\sigma_1 | \cdots |\sigma_p\big)$ such that $\mathsf{s}(\sigma_1)=a$ and $\mathsf{t}(\sigma_p)=b$, the coproducts above induce maps \[\nabla_\varepsilon \colon \mathbb{P}_\bullet(X)(a,b) \to \mathbb{P}_\bullet(X)(a,b) \otimes \mathbb{P}_\bullet(X)(a,b)\] for $\varepsilon=0,1$ and any $a,b \in X_0$. 
The coproducts $\nabla_0$ and $\nabla_1$ are appropriately compatible with the monoid structure of $\mathbb{P}_\bullet (X)$. Furthermore, the structure $(\mathbb{P}_\bullet(X), d, \mu, \nabla_0)$ has the following property: there is a chain map
\[S \colon \mathbb{P}_\bullet(X) \to \mathbb{P}_\bullet(X)\] extending $S([a,b])= ([b,a])$ and satisfying certain ``antipode equations" (Section \ref{SEC:S}). This structure may be understood as a generalization of a ``homotopy Hopf algebra" \cite[1.8]{B}. 

\subsection{}
Our second map $\chi: C_\bullet(X) \to \mathbb{L}_\bullet(X)$ is expressed in terms of $\Delta$, $\nabla_0$, $\nabla_1$, and $S$, as we now define. For any $\sigma \in X_k$, for simplicity write   
 \begin{eqnarray} \label{notationDelta}
     \Delta(\sigma)=\sigma' \otimes \sigma'' \in C_\bullet(X) \otimes C_\bullet(X)
\end{eqnarray}
and 
\begin{eqnarray} \label{notationnabla}
    \nabla_\varepsilon (\sigma)  = (\sigma)^{\varepsilon,1} \otimes (\sigma)^{\varepsilon,2} \in \mathbb{P}_\bullet(X) \otimes \mathbb{P}_\bullet(X)
 \end{eqnarray}
for $\varepsilon=0,1$. The coproduct $\nabla_0$ has the special property that when applied to $(\sigma) \in \mathbb{P}_{k-1}(X)$ we obtain a sum of tensors in which the second factor is always a sequence of simplices of length $1$. This means we can write each factor $(\sigma)^{0,2} \in \mathbb{P}_\bullet(X)$ as $(\sigma^{0,2})$ for some simplex $\sigma^{0,2}$ in $X$. Define a map $ \chi \colon C_\bullet(X) \to \mathbb{L}_\bullet(X)$ on any $\sigma \in X_0$ by $\chi (\sigma) = \big(\text{id}_{\sigma}\big) \sigma$, and, for $\sigma \in X_k$ where $k>0$, define $\chi(\sigma)$ by
\begin{multline}\label{EQU:chi}
\chi (\sigma) =(-1)^{|(\sigma)^{0,1}|} \big( S(( \sigma)^{0,1}) \big)   \sigma^{0,2} + 
\big( (\sigma)^{{1,1}} | S( (\sigma ) ^{1,2}) \big)  \mathsf{s}(\sigma) 
\\
+ (-1)^{ |\sigma'||\sigma''|+|(\sigma')^{0,1}|} \big(( \sigma'')^{1,1} | S(( \sigma'')^{1,2}) | S(( \sigma' )^{0,1})\big)  \sigma'^{\text{  }0,2}
\end{multline}
\begin{theorem}\label{Theorem2}
The map $\chi \colon (C_\bullet(X),\partial) \to (\mathbb{L}_\bullet(X),D)$ defined above is a chain map. On homology, the map coincides with the map induced by the constant loops embedding $\iota:|X| \hookrightarrow L|X|$.
\end{theorem}
We note that the map $\chi$ uses more than just the chain level coassociative coalgebra structure of $C_\bu(X)$: it depends on the $E_3$-structure of $C_\bu(X)$; in particular, $\nabla_0$ uses the $E_2$ structure and $\nabla_1$ the $E_3$-structure of $C_\bu(X)$. This can be seen, for example, when studying the string topology of a manifold, where a correct model for the string topology BV-algebra (over $\Z$ or a finite field) in terms of Hochschild homology may require the use of this higher structure. An instance of this phenomenon for the $2$-sphere was first noted by Menichi \cite{M}; see Example \ref{EXA:2-sphere} as well as \cite{PT} and \cite{GNC} for $\Z_2$-coefficients. 

Finally, we note that using standard arguments from acyclic models we also have the following theorem.
\begin{theorem}\label{Theorem-chain-homotopic}
The maps $\rho$ and $\chi$ are chain homotopic.
\end{theorem}

\section{The chain complex $(\mathbb{L}_\bullet(X), D)$} 

We give a detailed definition of the chain complex $(\mathbb{L}_\bullet(X),D)$ that was described above. This complex was defined in \cite{R} as an example of a more general construction, but below we give a direct and self-contained description.

Fix a commutative ring with unit $R$ and write $\otimes= \otimes_R$. For any graded set $S$ denote by $R\langle S \rangle$ the free graded $R$-module generated by $S$. Recall the underlying graded $R$-module of the chains in $X$ is $C_\bullet(X)= R\langle \bigcup_{j=0}^{\infty} X_j \rangle$. The first and last vertex maps $\mathsf{s}, \mathsf{t} \colon \bigcup_{j=0}^\infty X_j \to X_0$ induce $R$-linear maps
\[C_\bullet(X) \to C_0(X) \otimes C_\bullet(X)\]
and
\[ C_\bullet(X) \to C_\bullet(X) \otimes C_0(X) \]
endowing $C_\bullet(X)$ with the structure of a graded bicomodule over the coalgebra $(C_0(X)=R\langle X_0\rangle, \Delta)$, where $\Delta(a)=a\otimes a$ for $a\in X_0$. When equipped with the differential $\widetilde{\partial} \colon C_\bullet(X)\to C_{\bullet-1}(X)$ defined in \eqref{modifiedpartial}, $C_\bullet(X)$ becomes a dg $C_0(X)$-bicomodule. The Alexander-Whitney coproduct, considered as the map 
\[\Delta \colon C_\bullet(X) \to C_\bullet(X) \underset{C_0}{\square} C_\bullet(X)\]
\[\Delta([\bi_0,\dots, \bi_k])= \sum_{j=0}^k [\bi_0, \dots, \bi_j] \underset{C_0}{\square} [\bi_j, \dots, \bi_k], \]
 where $\underset{C_0}{\square}$ denotes the cotensor product over the coalgebra $C_0=(C_0(X), \Delta)$, is a map of dg $C_0(X)$-bicomodules. From now on, write $\square=\underset{C_0}{\square}$. Thus, the triple $(C_\bullet(X), \widetilde{\partial}, \Delta)$ defines a comonoid in the monoidal category of dg $C_0(X)$-bicomodules whose underlying $R$-module is free and monoidal structure given by $\square$.

\begin{definition}
For any graded $R$-module $V$ we let $\underline{V}$ be $V$ shifted down by $1$, i.e., $\underline{V}_j:=V_{j+1}$. Denote $C= C_{\bullet>0}(X) \oplus R \langle X_1^{-1} \rangle$, so that
$\unC$ is a non-negatively graded $C_0$-bicomodule. Consider the graded $R$-module \[\mathbb{T}^{C_0}_\bullet(C)= \bigoplus_{p=0}^{\infty} \underline{C}^{{\square } p}.\] Note $\mathbb{T}^{C_0}_\bullet(C)$ is generated by all sequences $\underline{\sigma_1} \square \cdots \square \underline{\sigma_p}$ where $\sigma_i \in \bigcup_{j=1}^\infty X_j \cup X_1^{-1}$. When $p=0$ the direct sum above is interpreted as a single copy of $C_0(X)$ and we denote the image of $a \in X_0$ in $\mathbb{T}^{C_0}_\bullet(C)$ by $\text{id}_a$. Furthermore, $\mathbb{T}^{C_0}_\bullet(C)$ inherits the structure of a graded $C_0(X)$-bicomodule. We quotient $\mathbb{T}^{C_0}_\bullet(C)$ by the relation generated by 
\[ \underline{\sigma_1} \square \cdots \square \underline{\sigma_i} \square \underline{\sigma_{i+1}} \square  \cdots \square \underline{\sigma_p} \sim \underline{\sigma_1} \square \cdots \square \underline{\sigma_{i-1}} \square \underline{\sigma_{i+2}} \square \cdots \square \underline{\sigma_p} \]
if either $\sigma_i \in X_1$ and $\sigma_{i+1}=\sigma_i^{-1}$ or $\sigma_{i+1}\in X_1$ and $\sigma_i = \sigma_{i+1}^{-1}$. Denote $\mathbb{P}_\bullet(X)= \mathbb{T}^{C_0}_\bullet(C)/\sim$ and the equivalence class of $\underline{\sigma_1} \square \cdots \square \underline{\sigma_p}$ by $\big(\sigma_1 | \cdots |\sigma_p \big)$. The graded $C_0(X)$-bicomodule structure passes to $\mathbb{P}_\bullet(X)$ and concatenation of sequences induces a map \[ \mu \colon \mathbb{P}_\bullet(X) \square \mathbb{P}_\bullet(X) \to \mathbb{P}_\bullet(X)\] making $\mathbb{P}_\bullet(X)$ into a monoid in the category of $C_0(X)$-bicomodules. 
\end{definition}
\begin{remark}
The monoid $(\mathbb{P}_\bullet(X),d, \mu)$ may  be understood as the localization of the cobar construction of the comonoid $(C_\bullet(X),\widetilde{\partial}, \Delta)$ along the multiplicative set generated by $X_1$. 
\end{remark}
\begin{definition} Define a graded $R$-module
\begin{equation}\label{necklacecomplex}
\mathbb{L}_\bullet(X)=\mathbb{P}_\bullet(X) \underset{C_0 \otimes C_0^{\text{op}}}{\square} C_{\bullet}(X).
\end{equation}
This is an algebraic description of the graded $R$-module generated by necklaces in $X$, see Section \ref{introduction}. Define a linear map $D \colon \mathbb{L}_\bullet(X) \to \mathbb{L}_{\bullet-1}(X)$ by
\begin{equation} \label{necklaceD}
D \,\,\,= \,\,\, d \underset{C_0 \otimes C_0^{op}}{\square} \text{id}_{C_\bullet(X)} 
 \,\,\,+ \,\,\, \text{id}_{\mathbb{P}_\bullet(X)} \underset{C_0 \otimes C_0^{op}}{\square} \widetilde{\partial}
  \,\,\, + \,\,\, \delta,
\end{equation} 
where $\wt\partial$ is the reduced differential \eqref{modifiedpartial}, and the maps $d$ and $\delta$ are described below.

Recall our notation of $C_\bullet(X)= R\langle \bigcup_{j=0}^{\infty} X_j \rangle$ and $C= C_{\bullet>0}(X) \oplus R \langle X_1^{-1} \rangle$ from above. In order to state the definitions of $d$ and $\delta$, we first need to define the reduced coproduct $\widetilde{\Delta} \colon C \to C \square C$, which is given by setting $\widetilde{\Delta}(\sigma)=0$ if $\sigma \in X_1 \cup X_1^{-1}$, and
\[\widetilde{\Delta}([\bi_0, \ldots, \bi_k]) = \sum_{j=1}^{k-1} =[\bi_0, \ldots, \bi_j] \square  [\bi_j, \ldots, \bi_k]
\]
for any $[\bi_0, \ldots, \bi_k]\in X_k$ with $k>1$. Denote by $s:C\to \unC$ the shift map $s(c)=\un{c}$, which is a map of degree $-1$. By abuse of notation, we also denote by $s \colon C_\bullet(X) \to \unC$ the map that is zero on $C_0(X)$ and the shift map on $C_{\bullet >0}(X)$. Then, we define:
\begin{align}
&\widetilde{\undel}:\unC\to\unC, &&\scalebox{1}{$\widetilde{\undel}(\un{c})=-s\circ \widetilde{\del} \circ s^{-1}(\un{c})=-\un{\widetilde{\del}(c)}$} \\
&\widetilde{\unDel}:\unC\to\unC \square \unC, &&  \widetilde{\unDel}(\un{c})= (s\ot s)\circ \widetilde{\Del} \circ s^{-1}(\un{c})={\sum_{(c)}}(-1)^{|\un{c'}|+1}\cdot\un{c'}\square \un{c''}
\end{align}
The map $d \colon \mathbb{P}_\bullet(X) \to \mathbb{P}_{\bullet-1}(X)$ that appears in the formula \eqref{necklaceD} is then defined by extending $\underline{\widetilde{\partial}}+ \underline{\widetilde{\Delta}}$ as a derivation of $\mu$. In fact, this makes $(\mathbb{P}_\bullet(X), d, \mu)$ a monoid in the category of dg $C_0$-bicomodules, which can be reformulated as a dg category structure on the object set $X_0$.

Finally to state the definition of $\delta$, we define the maps
\begin{align}
&\unDel_L:C_{\bullet}(X)\to\unC \square C_{\bullet}(X), &&  \hspace{-2mm}\unDel_L(c)=(s\ot 1)\circ \Del(c)={\sum_{(c)}}\,\,\,  \un{c'} \square c''\\
&\unDel_R:C_{\bullet}(X)\to C_{\bullet}(X)\square \unC, && \hspace{-2mm}
 \unDel_R(c)=-(1\ot s)\circ \Del(c)={\sum_{(c)}} (-1)^{|c'|+1}\cdot c' \square\un{c''}
\end{align}
and with this, $\delta$ is defined by setting
\begin{equation}\label{EQU:delta-in-P} 
 \delta \,\,\,=\,\,\, (\mu \underset{C_0 \otimes C_0^{op}}{\square} \text{id}_{C_\bullet(X)} ) \circ ( \text{id}_{\mathbb{P}_\bullet(X)} \underset{C_0 \otimes C_0^{op}}{\square} \underline{\Delta}_L)\,\,\,  +\,\,\,  \tau \circ (\text{id}_{\mathbb{P}_\bullet(X)} \underset{C_0 \otimes C_0^{op}}{\square}\underline{\Delta}_R),
\end{equation}
where 
\begin{multline*}
\tau \colon \big( \mathbb{P}_\bullet(X)  \underset{C_0 \otimes C_0^{op}}{\square}C_\bullet(X)\big) \underset{C_0 \otimes C_0^{op}}{\square} \unC \to \unC \square \big(\mathbb{P}_\bullet(X)  \underset{C_0 \otimes C_0^{op}}{\square} C_\bullet(X) \big)\\
\to \mathbb{P}_\bullet(X) \underset{C_0 \otimes C_0^{op}}{\square} C_\bullet(X)
\end{multline*}
is the switch map followed by the concatenation map $\unC \square \mathbb{P}_\bullet(X) \to \mathbb{P}_\bullet(X)$. We leave it as an exercise to check that $D^2=0$ and that the above definition $(\mathbb{L}_\bullet(X), D)$ agrees with the description given in the Section \ref{introduction}. The topological meaning of this construction is given by the following result.
\end{definition}

\begin{theorem}\label{necklacesandloopspace}\cite[Theorem 20]{R} For any ordered simplicial complex $X$, the chain complex $(\mathbb{L}_\bullet(X),D)$ is quasi-isomorphic to $(S_\bullet(L|X|), \partial)$ the singular chains on the free loop space of the geometric realization of $X$. 
\end{theorem}

\begin{remark} \label{relationtocoCH}
    We now relate $\mathbb{L}_\bullet(X)$ to the ordinary coHochschild and Hochschild complexes. We refer to \cite[Section 1]{HPS} for an exposition on the (co)Hochschild complex.  Suppose $X$ is a simply connected ordered simplicial complex.  Then there is a semi-simplicial set $Y$ with one $0$-simplex and no $1$-simplices together with a homotopy equivalence $\pi \colon X \xrightarrow{\simeq} Y$  of semi-simplicial sets. Let $C_\bullet(Y)$ be the dg coalgebra of  simplicial chains and denote by $coCH_\bullet(C_\bullet( Y), C_\bullet( Y))$ the (ordinary) coHochschild complex.  In particular, note that $C_0(Y)=R$ and $C_1(Y)=0$. Define a chain map 
\[ \theta_{\pi} \colon \mathbb{L}_\bullet(X) \to coCH_\bullet(C_\bullet(Y), C_\bullet(Y))\] by setting
    \begin{equation} \label{tocoHochschild}
 \theta_{\pi} (\big(\sigma_1 | \cdots |\sigma_p\big)\sigma_{p+1})= \big(\overline{\theta} (\sigma_1) | \cdots | \overline{\theta} (\sigma_p)\big)\pi(\sigma_{p+1}).
    \end{equation}
    where $\overline{\theta}(\sigma) = \pi(\sigma)$ if $\sigma \in X_k$ for $k>1$ and $\overline{\theta}(\sigma)=1$ if $\sigma \in X_1 \cup X_1^{-1}$. The map $\theta_{\pi}$ is a quasi-isomorphism (assuming that $X$ is simply connected). Furthermore, if $Y$ is a finite simplicial complex, then $coCH_\bullet(C_\bullet(Y), C_\bullet( Y))$ is exactly the linear dual of the Hochschild complex of the dg algebra of simplicial cochains $C^\bullet( Y)$, which, by \cite{J,U}, is quasi-isomorphic to the singular cochains on the free loop space of $|Y|$. 

\end{remark}

\begin{remark}\label{REM:CY}
    The data $(\mathbb{P}_\bullet(X), d, \mu)$ may be regarded as a dg category on the object set $X_0$. A generalization of a classical theorem of Adams tells us that this dg category is quasi-equivalent to the dg category of paths in $|X|$ \cite{A,R}.  Furthermore, an explicit natural chain homotopy equivalence between $\mathbb{L}_\bullet(X)$ and $CH_\bullet(\mathbb{P}_\bullet(X),\mathbb{P}_\bullet(X))$, the Hochschild chain complex of the dg category $\mathbb{P}_\bullet(X)$, is constructed in \cite{R}. This may be understood as an instance of dg Koszul duality. 
    
    When $X$ is a finite simplicial complex, the dg category $\mathbb{P}_\bullet(X)$ is smooth, i.e., perfect as a dg bimodule over itself. A weak smooth $n$-Calabi-Yau structure in $\mathbb{P}_\bullet(X)$ is a cycle in $CH_n(\mathbb{P}_\bullet(X),\mathbb{P}_\bullet(X))$ that satisfies an appropriate nondegeneracy condition (see \cite[Definition 12]{G2}). Hence, a  weak smooth $n$-Calabi-Yau structure on $\mathbb{P}_\bullet(X)$ can be specified by an $n$-cycle in the smaller complex $\mathbb{L}_\bullet(X)$.  
    
    Finiteness also implies that the dg algebra of cochains $C^\bullet(X)$ is proper, i.e., perfect as a dg $R$-module.  A weak proper (also called compact) $n$-Calabi-Yau structure on $C^\bullet(X)$ is a cycle in $CH^n(C^\bullet(X), C_\bullet(X))$, the Hochschild cochain complex of the dg algebra $C^\bullet(X)$ with coefficients in the $C^\bullet(X)$-bimodule $C_\bullet(X)$, satisfying an appropriate nondegeneracy condition (see \cite[Remark 49]{G2}). When $X$ is simply connected, $\mathbb{L}_\bullet(X)$ is quasi-isomorphic to $CH^\bullet(C^\bullet(X), C_\bullet(X))$, so, in this case, a weak proper $n$-Calabi-Yau structure on $C^\bullet(X)$ can also be specified by an $n$-cycle in $\mathbb{L}_\bullet(X)$. For more on Calabi-Yau structures, see \cite[Section 2.3]{KTV}.

\end{remark}

\section{A model for the constant loops map}

We now look at the map  $\rho:C_\bullet(X) \to \mathbb{L}_\bullet(X)$ defined in \eqref{EQU:rho}-\eqref{EQU:sigma[[s]]} in more detail by applying $\rho$ to low dimensional simplicies and we apply it to the $2$-sphere. We end with a proof of Theorem \ref{Theorem1}.

\subsection{Computing $\rho$ in examples}

Suppose $\sigma=[\bi_0,\dots, \bi_k]$ is a $k$-simplex in $X$. The next example computes $\rho(\sigma)$ in dimensions $k=0,1,2,3$.
\begin{example}\label{EXA:0123-simplicies}
We abbreviate a simplex $[\bi_0,\dots,\bi_k]$ by writing $\bi_0\dots\bi_k$.  Then, the formulas for the standard $0$-, $1$-, $2$-,  and $3$-simplicies are (with $\bi_0=\bo{0}$, $\bi_1=\bo{1}$, etc.):
\begin{align}
\rho(\bo{0})=&\bo{0}\\
\rho(\bo{0\bo{1}})=&\rA{10}{01}\\
\rho(\bo{0}\bo{1}\bo{2})=&\rA{20}{012}+ \rB{10}{012}{20}{01} +\rB{20}{012}{21}{12}\\
\nonumber & +\rC{10}{012}{20}{012}{21}{1}
\end{align}
The analogous formula for $\rho(\bo{012})$ in the coHochschild complex was also computed in \cite{PT} for $\Z_2$ coefficients.
\begin{align}
\rho(\bo{0123})=&\hspace{4mm} \rA{30}{0123} \\ \nonumber
& -\rB{20}{023}{30}{012}\\ \nonumber
&+\rB{30}{013}{31}{123}\\ \nonumber
& -\rB{10}{0123}{30}{01}\\ \nonumber
&-\rB{20}{0123}{31}{12}\\ \nonumber
& -\rB{30}{0123}{32}{23}\\ \nonumber
& +\rC{10}{012}{20}{023}{30}{01}\\ \nonumber
&+\rC{20}{012}{21}{123}{31}{12}\\ \nonumber
& -\rC{20}{023}{30}{013}{31}{12}\\ \nonumber
&+\rC{30}{013}{31}{123}{32}{23}\\ \nonumber
& +\rC{10}{012}{20}{0123}{31}{1}\\ \nonumber
&-\rC{10}{0123}{30}{013}{31}{1}\\ \nonumber
& +\rC{20}{023}{30}{0123}{32}{2}\\ \nonumber
& -\rC{20}{0123}{31}{123}{32}{2}\\ \nonumber
& +\rD{10}{012}{20}{023}{30}{013}{31}{1}\\ \nonumber
& -\rD{10}{012}{20}{012}{21}{123}{31}{1}\\ \nonumber
& -\rD{20}{023}{30}{013}{31}{123}{32}{2}\\ \nonumber
& +\rD{20}{012}{21}{123}{31}{123}{32}{2}
\end{align}
The above terms are all coming from the indexing set $\Ss^{({\ell})}_3$ described in \eqref{EQU:S^l_k}. For example, the term $\rD{10}{012}{20}{012}{21}{123}{31}{1}$ comes from the indexing sequence \[s=((0,1),(0,2),(1,2),(1,3))\in \Ss^{({4})}_3,\] which is easily seen to satisfy the conditions in \eqref{EQU:S^l_k}. One can visualize this sequence $s$ as a sequence of brackets starting at $(0,q_1)=(0,1)$, ending at $(p_\ell,k)=(1,3)$, which keep moving to the right, while overlapping at consecutive stages (since \eqref{EQU:S^l_k} requires $p_{j+1}<q_j$).
\[
\begin{tikzpicture}[xscale=0.5,yscale=0.5]
\node [left] at (-1.5,1) {visualization of $((0,1),(0,2),(1,2),(1,3)):$};
\node [right] at (0,2) {$0$}; \node [right] at (1,2) {$1$};
\node [right] at (2,2) {$2$}; \node [right] at (3,2) {$3$};
\draw (0.5,1.5)--(0.5,1.3)--(1.5,1.3)--(1.5,1.5);
\draw (0.5,1.1)--(0.5,0.9)--(2.5,0.9)--(2.5,1.1);
\draw (1.5,0.7)--(1.5,0.5)--(2.5,0.5)--(2.5,0.7);
\draw (1.5,0.3)--(1.5,0.1)--(3.5,0.1)--(3.5,0.3);
\end{tikzpicture}
\]
\end{example}

\begin{remark}\label{RMK:rho}
To illustrate the combinatorics of the $\rho$ map, we give an example of how to choose the terms in the output $\rho([\bi_0,\dots,\bi_{14}])$ of a $14$-dimensional simplex with vertices $\bi_0,\dots,\bi_{14}$. In fact, by \eqref{EQU:rho}, one needs to sum over all sequence $s=((p_1,q_1),\dots,(p_\ell,q_\ell))\in \Ss^{(\ell)}_{14}$ for all $\ell$. For example, when $\ell=5$, we need all $s=((p_1,q_1),(p_2,q_2),(p_3,q_3),(p_4,q_4),(p_5,q_5))\in \Ss^{(5)}_{14}$. Note from \eqref{EQU:S^l_k}, that to get a sequence $s$, we need to choose numbers in non-decreasing order:
\[
0=p_1\leq p_2\leq p_3\leq p_4\leq p_5\leq q_1\leq q_2\leq q_3\leq q_4\leq q_5=14
\]
\begin{itemize}
\item
Now, if we start, for example, by choosing $p_5$ and $q_1$, this will determine the special tensor on the right; say, if we pick $p_5=6$ and $q_1=9$, then we get a term
\[
\big(\quad\dots\quad\big) [\bi_6,\bi_7,\bi_8,\bi_9]
\]
(Note that we do allow the case where $p_\ell=q_1$.)
\item
Next, if we pick, for example, $p_2=4$ and $q_2=11$, the first two tensor factors in $\mathbb{L}_\bu(X)$ become (since $p_1=0$ and $q_1=9$, the sequence $\bi_{p_1},\dots, \bi_{p_2}$ is $\bi_0,\bi_1,\bi_2,\bi_3,\bi_4$, and $\bi_{q_1},\dots, \bi_{q_2}$ is $\bi_9,\bi_{10},\bi_{11}$):
\[
\big([\bi_9,\bi_0]\big| [\bi_0,\bi_1,\bi_2,\bi_3,\bi_4,\bi_9,\bi_{10},\bi_{11}] \big|\quad\dots\quad\big) [\bi_6,\bi_7,\bi_8,\bi_9]
\]
\item
Next, say, we pick, $p_3=5$ and $q_3=13$, we get the next two tensor factors:
\[
\big(
[\bi_9,\bi_0]\big| [\bi_0,\bi_1,\bi_2,\bi_3,\bi_4,\bi_9,\bi_{10},\bi_{11}] \big|
[\bi_{11},\bi_4]\big| [\bi_4,\bi_5,\bi_{11},\bi_{12},\bi_{13}] \big|
\quad\dots\quad\big) [\bi_6,\bi_7,\bi_8,\bi_9]
\]
\item
Continuing in this way we may pick, for example, the sequence $(p_1,q_1)=(0,9)$, $(p_2,q_2)=(4,11)$, $(p_3,q_3)=(5,13)$, $(p_4,q_4)=(5,14)$ and $(p_5,q_5)=(6,14)$. This gives:
\begin{multline*}
\big(
[\bi_9,\bi_0]\big| [\bi_0,\bi_1,\bi_2,\bi_3,\bi_4,\bi_9,\bi_{10},\bi_{11}] 
\big|[\bi_{11},\bi_4]\big| [\bi_4,\bi_5,\bi_{11},\bi_{12},\bi_{13}] 
\\
\big|[\bi_{13},\bi_5]\big| [\bi_5,\bi_{13},\bi_{14}] 
\big|[\bi_{14},\bi_5]\big| [\bi_5,\bi_6,\bi_{14}] \big|
[\bi_{14},\bi_6]
\big) [\bi_6,\bi_7,\bi_8,\bi_9]
\end{multline*}
\[
\begin{tikzpicture}[xscale=0.5,yscale=0.5]
\node [right] at (0,2) {$\bi_0$}; \node [right] at (1,2) {$\bi_1$}; \node [right] at (2,2) {$\bi_2$};
\node [right] at (3,2) {$\bi_3$}; \node [right] at (4,2) {$\bi_4$}; \node [right] at (5,2) {$\bi_5$};
\node [right] at (6,2) {$\bi_6$}; \node [right] at (7,2) {$\bi_7$}; \node [right] at (8,2) {$\bi_8$};
\node [right] at (9,2) {$\bi_9$}; \node [right] at (10,2) {$\bi_{10}$}; \node [right] at (11,2) {$\bi_{11}$};
\node [right] at (12,2) {$\bi_{12}$}; \node [right] at (13,2) {$\bi_{13}$}; \node [right] at (14,2) {$\bi_{14}$};
\draw (0.5,1.5)--(0.5,1.3)--(9.5,1.3)--(9.5,1.5);
\draw (4.5,1.1)--(4.5,0.9)--(11.5,0.9)--(11.5,1.1);
\draw (5.5,0.7)--(5.5,0.5)--(13.5,0.5)--(13.5,0.7);
\draw (5.5,0.3)--(5.5,0.1)--(14.5,0.1)--(14.5,0.3);
\draw (6.5,-0.1)--(6.5,-0.3)--(14.5,-0.3)--(14.5,-0.1);
\draw (0.5,2.5)--(0.5,3); \draw (4.5,2.5)--(4.5,3); 
\draw (5.5,2.5)--(5.5,3); \draw (6.5,2.5)--(6.5,3); 
\draw (9.5,2.5)--(9.5,3); \draw (11.5,2.5)--(11.5,3); 
\draw (13.5,2.5)--(13.5,3); \draw (14.5,2.5)--(14.5,3); 
\node at (0.5,3.5) {$\bi_{p_1}$}; \node at (9.5,3.5) {$\bi_{q_1}$};
\node at (4.5,3.5) {$\bi_{p_2}$}; \node at (11.5,3.5) {$\bi_{q_2}$};
\node at (5.4,4.1) {$\bi_{p_3}$}; \node at (13.5,3.5) {$\bi_{q_3}$};
\node at (5.6,3.4) {$\bi_{p_4}$}; \node at (14.4,4.1) {$\bi_{q_4}$};
\node at (6.5,3.5) {$\bi_{p_5}$}; \node at (14.6,3.4) {$\bi_{q_5}$};
\end{tikzpicture}
\]
\end{itemize}
The conditions $p_{j+1}<q_{j}$ guarantee that the simplex $[\bi_{p_j},\dots,\bi_{p_{j+1}},\bi_{q_j},\dots,\bi_{q_{j+1}}]$ is non-degenerate, (i.e, is a simplex of the correct geometric dimension $(p_{j+1}-p_j)+(q_{j+1}-q_j)+1$, which by $p_j+q_j<p_{j+1}+q_{j+1}$ is greater than $1$). This disallows, for example, sequences such as $((0,5),(5,14))$. Moreover, we have that $p_j<q_j$, which follows from the conditions $p_{j+1}<q_{j}$ and $0=p_1\leq p_2\leq \dots\leq p_\ell\leq q_1\leq q_2\leq \dots \leq q_\ell=k$; this guarantees that the simplex $[\bi_{q_j},\bi_{p_j}]$ is a (non-degenerate) $1$-simplex.
\end{remark}

\begin{example}\label{EXA:2-sphere}
We now apply the map $\rho$ to the $2$-sphere $S^2$ considered as the boundary of the $3$-simplex $[\bo{0},\bo{1},\bo{2},\bo{3}]$ with $R=\mathbb{Z}$. This extends the example of the $2$-sphere given in \cite{PT} for $\Z_2$. Since the chain $c=[\bo{0},\bo{1},\bo{2}]-[\bo{0},\bo{1},\bo{3}]+[\bo{0},\bo{2},\bo{3}]-[\bo{1},\bo{2},\bo{3}]\in C_2(S^2)$ is closed, we get a closed element $\rho(c)\in \mathbb{L}_2(X)$. Consider the semi-simplicial set $S^2/\mathcal{T}$ obtained by collapsing the set of $2$-simplices $\mathcal{T}=\{[0,1,3], [0,2,3], [1,2,3]\}$ to a point and denote by $\pi \colon S^2\to S^2/\mathcal{T}$ the quotient map, which is a homotopy equivalence. The dg coalgebra of simplicial chains $C_\bullet(S^2/\mathcal{T})$ is isomorphic to the homology of the $2$-sphere $H=H_\bu(S^2)=\Z.\sigma\oplus \Z.\varepsilon$ with $|\varepsilon|=0$ and $|\sigma|=2$, trivial differential, and coproduct determined by $\Delta(\varepsilon)=\varepsilon \otimes \varepsilon$, $\Delta(\sigma)=\varepsilon \otimes \sigma + \sigma \otimes \varepsilon$.
Using the notation of Remark \ref{relationtocoCH} and the computation for the $2$-simplex from Example \ref{EXA:0123-simplicies}, we obtain
\[
\theta_{\pi} (\rho(c))= \big(\text{id}\big)\pi[012] + \big(\pi[012]| \pi[012]\big)\pi[0]= \big(\text{id}\big)\sigma+ \big(\sigma | \sigma\big)\varepsilon.
\]
We note that since $H$ is graded co-commutative, $ \big(\text{id}\big)\sigma\in coCH_2(H,H)$ by itself is also closed. In fact, we also have a chain map $H\to coCH_\bu(H,H), h\mapsto h$, and so we get a map $H=H_\bullet(S^2) \to H_\bu(coCH_\bu(H,H))\cong H_\bu(L|S^2|)$. McGowan, Naef, and O'Callaghan showed in \cite{GNC} this map does \emph{not} model the inclusion of constant loops.
\end{example}

\subsection{Proof of Theorem \ref{Theorem1}} 

In this subsection we prove Theorem \ref{Theorem1}.

First, we show that  $\rho:C_\bullet(X) \to \mathbb{L}_\bullet(X)$ is a chain map, i.e., we show that $\rho(\del([\bi_0,\dots,\bi_k]))=D(\rho([\bi_0,\dots,\bi_k]))$, where $D$ is the differential defined in \eqref{necklaceD}. This will be done via a straightforward but lengthy computation. The left-hand side of our equation gives:
\begin{multline}\label{EQU:rho-d=}
\rho(\del([\bi_0,\dots,\bi_k]))\\
=
\ub{\rho([\bi_1,\dots,\bi_k])}_{\lb{EQU:rho-d=}{A}}
+\ub{\sum_{0<j<k}(-1)^j \rho([\bi_0,\dots,\wh{\bi_j},\dots,\bi_k])}_{\lb{EQU:rho-d=}{B}}
+\ub{(-1)^k \rho([\bi_0,\dots,\bi_{k-1}])}_{\lb{EQU:rho-d=}{C}}
\end{multline}
where $\wh{\bi_j}$ denotes that the index has been removed. The term in \lb{EQU:rho-d=}{A} appears in \eqref{EQU:sum_j-removeL} below. The terms in \lb{EQU:rho-d=}{B} are those where an index in the middle is missing; these terms are precisely those that appear in \lb{EQU:D(pq..q)}{B}, \lb{EQU:D(p..pq)}{A}, \lb{EQU:D(p..pq..q)}{A}, \lb{EQU:D(p..pq..q)}{D}, and \eqref{EQU:del-last} below. The term in \lb{EQU:rho-d=}{C} appears in \eqref{EQU:sum_j-removeR} below.

We now compute $D(\rho([\bi_0,\dots,\bi_k]))$ and check that it agrees with \eqref{EQU:rho-d=}. By \eqref{necklaceD}, $D$ operates on the tensor factors in \eqref{EQU:rho} in $\mathbb{P}_\bullet(X)$ by extending $\un{\wt{\partial}}+ \un{\wt{\Delta}}$ to a derivation.
\begin{itemize}
\item[$\blacksquare$]
For $r=1,\dots, \ell$, applying the differential to a tensor factor $\un{[\bi_{q_r},\bi_{p_r}]}$ gives zero, since $\underline{\wt{\partial}}$ and $\underline{\wt{\Delta}}$ vanish on ``edges'', i.e., $(\un{\wt{\partial}}+ \un{\wt{\Delta}})(\un{[\bi_{q_r},\bi_{p_r}]})=0$.
\item[$\blacksquare$]
Next, for $r=1,\dots, \ell-1$, we apply $D$ to $\un{[\bi_{p_{r}},\dots,\bi_{p_{r+1}},\bi_{q_{r}},\dots,\bi_{q_{r+1}}]}$.\\
Again, from \eqref{necklaceD} we apply $\un{\wt{\partial}}$ and $\un{\wt{\Delta}}$ to this term.
We need to distinguish five cases depending on whether $p_r$ and $p_{r+1}$ are equal, differ by one, or differ by more than one, and the same with $q_r$ and $q_{r+1}$.
\begin{itemize}
\item[{\bf Case 1:}] \quad $\un{[\bi_{p_{r}},\bi_{q_{r}},\bi_{q_{r+1}}]}$ % ---------------------(CASE 1)-------
   \tabto{4.8cm} (Thus: $p_r=p_{r+1}$ and $q_r+1=q_{r+1}$.)
\begin{equation}\label{EQU:D(pqq)}
(\un{\wt{\partial}}+ \un{\wt{\Delta}})(\un{[\bi_{p_{r}},\bi_{q_{r}},\bi_{q_{r+1}}]})
=
\ub{\un{[\bi_{p_{r}},\wh{\bi_{q_{r}}},\bi_{q_{r+1}}]}}_{\lb{EQU:D(pqq)}{A}}
-\ub{\un{[\bi_{p_{r}},\bi_{q_{r}}]}\ot \un{[\bi_{q_{r}},\bi_{q_{r+1}}]}}_{\lb{EQU:D(pqq)}{B}}
\end{equation}
\item[{\bf Case 2:}]\quad $\un{[\bi_{p_{r}},\bi_{q_{r}},\dots,\bi_{q_{r+1}}]}$  % -------------(CASE 2)-------
   \tabto{4.8cm} (Thus: $p_r=p_{r+1}$ and $q_r+1<q_{r+1}$.)
\begin{align}\label{EQU:D(pq..q)} 
&\hspace{-2mm} (\un{\wt{\partial}}+ \un{\wt{\Delta}})(\un{[\bi_{p_{r}},\bi_{q_{r}},\dots, \bi_{q_{r+1}}]})
\\ \nonumber &\hspace{-2mm}=
\ub{\un{[\bi_{p_{r}},\wh{\bi_{q_{r}}},\dots,\bi_{q_{r+1}}]}}_{\lb{EQU:D(pq..q)}{A}}
+\ub{\scalebox{0.95}[1]{$\sum\limits_{q_r<j<q_{r+1}}(-1)^{j-q_r}\un{[\bi_{p_{r}},\bi_{q_{r}},\dots, \wh{\bi_{j}}\dots, \bi_{q_{r+1}}]}$}}_{\lb{EQU:D(pq..q)}{B}}
\\ \nonumber & \hspace{-1mm}
-\ub{\un{[\bi_{p_{r}},\bi_{q_{r}}]}\ot \un{[\bi_{q_{r}},\dots,\bi_{q_{r+1}}]}}_{\lb{EQU:D(pq..q)}{C}}
+\ub{\scalebox{0.9}[1]{$\sum\limits_{q_r<j<q_{r+1}}(-1)^{j-q_r+1}\un{[\bi_{p_{r}},\bi_{q_{r}},\dots, \bi_{j}]}\ot \un{[\bi_{j},\dots,\bi_{q_{r+1}}]}$}}_{\lb{EQU:D(pq..q)}{D}}
\end{align} 
\item[{\bf Case 3:}]\quad $\un{[\bi_{p_{r}},\bi_{p_{r+1}},\bi_{q_{r}}]}$   % ----------------(CASE 3)-------
   \tabto{4.8cm} (Thus: $p_r+1=p_{r+1}$ and $q_r=q_{r+1}$.)
   \begin{equation}\label{EQU:D(ppq)}
(\un{\wt{\partial}}+ \un{\wt{\Delta}})(\un{[\bi_{p_{r}},\bi_{p_{r+1}},\bi_{q_{r}}]})
=
\ub{\un{[\bi_{p_{r}},\wh{\bi_{p_{r+1}}},\bi_{q_{r}}]}}_{\lb{EQU:D(ppq)}{A}}
-\ub{\un{[\bi_{p_{r}},\bi_{p_{r+1}}]}\ot \un{[\bi_{p_{r+1}},\bi_{q_{r}}]}}_{\lb{EQU:D(ppq)}{B}}
\end{equation}
\item[{\bf Case 4:}] \quad$\un{[\bi_{p_{r}},\dots,\bi_{p_{r+1}},\bi_{q_{r}}]}$  % -----------(CASE 4)-------
   \tabto{4.8cm} (Thus: $p_r+1<p_{r+1}$ and $q_r=q_{r+1}$.)
   \begin{align}\label{EQU:D(p..pq)} 
&\hspace{-1mm} (\un{\wt{\partial}}+ \un{\wt{\Delta}})(\un{[\bi_{p_{r}},\dots,\bi_{p_{r+1}},\bi_{q_{r}}]})
\\ \nonumber &
\hspace{-1mm}= 
\ub{\scalebox{0.8}[1]{$(-1)^{j-p_r+1}\sum\limits_{p_r<j<p_{r+1}}\un{[\bi_{p_{r}},\dots,\wh{\bi_{j}},\dots,\bi_{p_{r+1}},\bi_{q_{r}}]}$}}_{\lb{EQU:D(p..pq)}{A}}
+\ub{(-1)^{p_{r+1}-p_r+1}\un{[\bi_{p_{r}},\dots,\wh{\bi_{p_{r+1}}},\bi_{q_{r}}]}}_{\lb{EQU:D(p..pq)}{B}}
\\ \nonumber & \hspace{0mm}
+\ub{\scalebox{0.8}[1]{$\sum\limits_{p_r<j<p_{r+1}}(-1)^{j-p_r}\un{[\bi_{p_{r}},\dots, \bi_{j}]}\ot \un{[\bi_{j},\dots,\bi_{p_{r+1}},\bi_{q_{r}}]}$}}_{\lb{EQU:D(p..pq)}{C}}
%\\ \nonumber & \hspace{.4cm}
+\ub{\scalebox{0.8}[1]{$(-1)^{p_{r+1}-p_r}\un{[\bi_{p_{r}},\dots,\bi_{p_{r+1}}]}\ot \un{[\bi_{p_{r+1}},\bi_{q_{r}}]}$}}_{\lb{EQU:D(p..pq)}{D}}
\end{align}
\item[{\bf Case 5:}]\quad $\un{[\bi_{p_{r}},\dots,\bi_{p_{r+1}},\bi_{q_{r}},\dots,\bi_{q_{r}}]}$  
                                                                                                                      % ----------------(CASE 5)-------
   \tabto{4.8cm} (Thus: $p_r<p_{r+1}$ and $q_r<q_{r+1}$.)
%\begin{equation}\label{EQU:D(p..pq..q)} \hspace{-78mm}(\un{\wt{\partial}}+ \un{\wt{\Delta}})(\un{[\bi_{p_{r}},\ds, \bi_{p_{r+1}},\bi_{q_{r}},\ds,\bi_{q_{r+1}}]})\end{equation}
\begin{align}\label{EQU:D(p..pq..q)} 
\hspace{10mm} & \hspace{-9mm}(\un{\wt{\partial}}+ \un{\wt{\Delta}})(\un{[\bi_{p_{r}},\ds, \bi_{p_{r+1}},\bi_{q_{r}},\ds,\bi_{q_{r+1}}]})
\\
\nonumber
= &
\ub{\scalebox{1}[1]{$\sum\limits_{p_r<j<p_{r+1}}(-1)^{j-p_r+1}\un{[\bi_{p_{r}},\ds, \wh{\bi_{j}},\ds, \bi_{p_{r+1}},\bi_{q_{r}},\ds,\bi_{q_{r+1}}]}$}}_{\lb{EQU:D(p..pq..q)}{A}} 
\\ \nonumber & 
+\ub{(-1)^{p_{r+1}-p_r+1}\un{[\bi_{p_{r}},\ds, \wh{\bi_{p_{r+1}}},\bi_{q_{r}},\ds,\bi_{q_{r+1}}]}}_{\lb{EQU:D(p..pq..q)}{B}}
\\ \nonumber &
+\ub{(-1)^{p_{r+1}-p_r}\un{[\bi_{p_{r}},\ds, \bi_{p_{r+1}},\wh{\bi_{q_{r}}},\ds,\bi_{q_{r+1}}]}}_{\lb{EQU:D(p..pq..q)}{C}} 
\\ \nonumber &
+\ub{\sum_{q_r<j<q_{r+1}}(-1)^{p_{r+1}-p_r+j-q_r}\un{[\bi_{p_{r}},\ds, \bi_{p_{r+1}},\bi_{q_{r}},\ds,\wh{\bi_{j}},\ds\bi_{q_{r+1}}]}}_{\lb{EQU:D(p..pq..q)}{D}}
\\ \nonumber & 
+\ub{\sum_{p_r<j\leq p_{r+1}}(-1)^{j-p_r}\un{[\bi_{p_{r}},\ds,\bi_{j}]}\ot \un{[\bi_j,\ds,\bi_{p_{r+1}},\bi_{q_{r}},\ds,\bi_{q_{r+1}}]}}_{\lb{EQU:D(p..pq..q)}{E}}
\\ \nonumber & 
+\ub{\sum_{q_r\leq j<q_{r+1}}(-1)^{p_{r+1}-p_r+j-q_r+1}\un{[\bi_{p_{r}},\ds,\bi_{p_{r+1}},\bi_{q_{r}},\ds, \bi_{j}]}\ot \un{[\bi_{j},\ds,\bi_{q_{r+1}}]}}_{\lb{EQU:D(p..pq..q)}{F}}
\end{align}
\end{itemize}
We next identify terms that combine and cancel with other terms.

$\bcr$ First, the terms where an internal index is removed appears in the formula \lb{EQU:rho-d=}{B}. These are the terms labeled \lb{EQU:D(pq..q)}{B}, \lb{EQU:D(p..pq)}{A}, \lb{EQU:D(p..pq..q)}{A}, and \lb{EQU:D(p..pq..q)}{D}.

$\bcr$ Next, we note that \lb{EQU:D(pqq)}{B} combines with the tensors to its left and right; more precisely, the term $\un{[\bi_{q_r}, \bi_{p_r}]}\ot $\lb{EQU:D(pqq)}{B}$\ot \un{[\bi_{q_{r+1}}, \bi_{p_{r+1}}]}$ simplifies to $-\un{[\bi_{q_r}, \bi_{q_{r+1}}]}\ot  \un{[\bi_{q_{r+1}}, \bi_{p_{r+1}}]}$, where we used the fact that tensor products of formal inverses of edges cancel. Similarly, if we apply the tensors to the left and right of \lb{EQU:D(pq..q)}{C}, we obtain $\un{[\bi_{q_r}, \bi_{p_r}]}\ot $\lb{EQU:D(pq..q)}{C}$\ot \un{[\bi_{q_{r+1}}, \bi_{p_{r+1}}]}=-\un{[\bi_{q_r}, \dots,\bi_{q_{r+1}}]}\ot \un{[\bi_{q_{r+1}}, \bi_{p_{r+1}}]}$. Thus for any $p_r=p_{r+1}$ and $q_r< q_{r+1}$ we obtain the new factors:
\begin{equation}\label{EQU:qs-on-the-left}
-\un{[\bi_{q_{r}},\dots, \bi_{q_{r+1}}]}\ot \un{[\bi_{q_{r+1}}, \bi_{p_{r+1}}]}
\end{equation}
(which is obtained from the expression $\un{[\bi_{q_r}, \bi_{p_r}]}\ot x \ot \un{[\bi_{p_{r+1}}, \bi_{q_{r+1}}]}$, where $x=$ \lb{EQU:D(pqq)}{B} for $q_r+1= q_{r+1}$, and $x=$ \lb{EQU:D(pq..q)}{C} for $q_r+1< q_{r+1}$).
Thus, in \eqref{EQU:qs-on-the-left} we effectively split off a tensor $\un{[\bi_{q_{r}},\dots, \bi_{q_{r+1}}]}$ (all with $q$ indices) on the left. These are the precisely the terms appearing in \lb{EQU:D(pq..q)}{D} and \lb{EQU:D(p..pq..q)}{F} when $\un{[\bi_{q_{r}},\dots, \bi_{q_{r+1}}]}$  is not the first tensor factor, and  \eqref{EQU:Del_R-last} (where $\un{[\bi_{j},\dots, \bi_{q_{1}}]}$ is rotated to the front) when $\un{[\bi_{q_{r}},\dots, \bi_{q_{r+1}}]}$ is the first tensor factor.

$\bcr$ Similarly, we can combine \lb{EQU:D(ppq)}{B}, or, respectively \lb{EQU:D(p..pq)}{D}, with the tensor factors before and after them (using that in this case $\bi_{q_r}=\bi_{q_{r+1}}$). With this, for any $p_r<p_{r+1}$ and $q_r=q_{r+1}$, we can rewrite the expression $ \un{[\bi_{q_r}, \bi_{p_r}]}\ot x \ot  \un{[\bi_{q_{r+1}}, \bi_{p_{r+1}}]}$ as:
\begin{equation}\label{EQU:ps-on-the-right}
(-1)^{p_{r+1}-p_r} \cdot \un{[\bi_{q_{r}}, \bi_{p_{r}}]} \ot \un{[\bi_{p_{r}},\dots, \bi_{p_{r+1}}]}
\end{equation}
where $x$ is either $x=$ \lb{EQU:D(ppq)}{B} for $p_r+1=p_{r+1}$, and $x=$ \lb{EQU:D(p..pq)}{D} for $p_r+1<p_{r+1}$. In \eqref{EQU:ps-on-the-right}, we thus split off a factor of $\un{[\bi_{p_{r}},\dots, \bi_{p_{r+1}}]}$ (using the ``$p$ indices'') on the right. These are the precisely the terms appearing in \lb{EQU:D(p..pq)}{C} and \lb{EQU:D(p..pq..q)}{E} when $\un{[\bi_{p_{r}},\dots, \bi_{p_{r+1}}]}$  is not the last tensor factor before $[\bi_{p_\ell},\dots,\bi_{q_1}]$, and \eqref{EQU:Del_L-last} when $\un{[\bi_{p_{r}},\dots, \bi_{p_{r+1}}]}$ is the last tensor factor before $[\bi_{p_\ell},\dots,\bi_{q_1}]$.

$\bcr$ Next, we also combine \lb{EQU:D(pqq)}{A} with its tensors appearing right before and after it. Using that we are in the case of $p_r=p_{r+1}$ and $q_r+1=q_{r+1}$, and the fact that formal inverses of edges cancel, we get that:
\begin{multline}\label{EQU:sum_j-removeR}
\un{[\bi_{q_r}, \bi_{p_r}]}\ot \text{\lb{EQU:D(pqq)}{A}}\ot \un{[\bi_{q_{r+1}}, \bi_{p_{r+1}}]}
\\
=\un{[\bi_{q_r}, \bi_{p_r}]}\ot \un{[\bi_{p_{r+1}},\bi_{q_{r+1}}]}\ot \un{[\bi_{q_{r+1}}, \bi_{p_{r+1}}]}
= \un{[\bi_{q_r}, \bi_{p_r}]}
\end{multline}
Note, that since $p_r=p_{r+1}$ and $q_r+1=q_{r+1}$, the tensor factor following \eqref{EQU:sum_j-removeR} (which is $\un{[\bi_{p_{r+1}},\dots,\bi_{p_{r+2}},\bi_{q_{r+1}},\dots, \bi_{q_{r+2}}]}$ when the tensor is not next to the last) does not give a term of the form required by our formula \eqref{EQU:rho}, since it has an index $\bi_{q_r}$ missing. These are precisely the terms appearing and canceling with \lb{EQU:D(pq..q)}{A} and \lb{EQU:D(p..pq..q)}{C}. When the tensor is on the next to last spot (i.e., $r=\ell-1$), we necessarily have $q_{r+1}=k$ (see \eqref{EQU:S^l_k}), so that missing the corresponding index gives precisely the terms from \lb{EQU:rho-d=}{C} where $\bi_{k}$ is missing.

$\bcr$ Finally, we combine \lb{EQU:D(ppq)}{A} with its left and right tensor factors, just as we did in the last bullet point. Here, we have that $p_r+1=p_{r+1}$ and $q_r=q_{r+1}$, and we get:
\begin{multline}\label{EQU:sum_j-removeL}
\un{[\bi_{q_r}, \bi_{p_r}]}\ot \text{\lb{EQU:D(ppq)}{A}}\ot \un{[\bi_{q_{r+1}}, \bi_{p_{r+1}}]}
\\
=\un{[\bi_{q_r}, \bi_{p_r}]}\ot \un{[\bi_{p_{r}},\bi_{q_{r}}]}\ot \un{[\bi_{q_{r+1}}, \bi_{p_{r+1}}]}
= \un{[\bi_{q_{r+1}}, \bi_{p_{r+1}}]}
\end{multline}
Now, still using that $p_r+1=p_{r+1}$ and $q_r=q_{r+1}$, the tensor factor right before \lb{EQU:sum_j-removeL}{B} (that is, $\un{[\bi_{p_{r-1}},\dots,\bi_{p_{r}},\bi_{q_{r-1}},\dots, \bi_{q_{r}}]}$ when the tensor is not the first factor) misses $\bi_{p_{r+1}}$ when compared to the formula in \eqref{EQU:rho}. This term thus cancels with \lb{EQU:D(p..pq)}{B} and \lb{EQU:D(p..pq..q)}{B}, unless \eqref{EQU:sum_j-removeL} appears in the first tensor factor. In the case where \eqref{EQU:sum_j-removeL} does appear as the first tensor factor (i.e., when $r=1$), we have $p_1=0$ (by \eqref{EQU:S^l_k}), and thus our first tensor \eqref{EQU:sum_j-removeL} is $\un{[\bi_{q_1},\bi_1]}$. It follows that in this case the index $\bi_0$ is missing, which are precisely the terms in \lb{EQU:rho-d=}{A}.
\item[$\blacksquare$]
Lastly, we apply $D$ to the last tensor factor in \eqref{EQU:rho}, i.e., to $[\bi_{p_\ell},\dots,\bi_{q_1}]$.\\
From \eqref{necklaceD} and \eqref{EQU:delta-in-P} we see that $D$ applies to the last tensor factor  via $\wt{\del}$, $\unDel_L$, and $\unDel_R$ (where, after taking $\unDel_L$ and $\unDel_R$, the shifted tensor factors are combined into $\mathbb{P}_\bullet(X)$, and for $\unDel_R$ there is also a cyclic rotation of the tensor factors; cf. \eqref{necklaceD} and \eqref{EQU:delta-in-P}).
\begin{align}
\label{EQU:del-last}&\wt{\del}([\bi_{p_\ell},\dots,\bi_{q_1}])
=\sum_{p_\ell<j<q_1} (-1)^{j-p_\ell} [\bi_{p_\ell},\dots,\wh{\bi_{j}},\dots, \bi_{q_1}] 
\\
\label{EQU:Del_L-last}&\unDel_L([\bi_{p_\ell},\dots,\bi_{q_1}])
=
\sum_{p_\ell<j\leq q_1} \un{[\bi_{p_\ell},\dots,\bi_{j}]}\ot [\bi_{j},\dots,\bi_{q_1}]
\\
\label{EQU:Del_R-last}&\unDel_R([\bi_{p_\ell},\dots,\bi_{q_1}])
=\sum_{p_\ell\leq j<q_1} (-1)^{j-p_\ell+1} [\bi_{p_\ell},\dots,\bi_{j}]\ot \un{[\bi_{j},\dots,\bi_{q_1}]}
\end{align}

The terms in \eqref{EQU:del-last} (missing $\bi_{j}$ altogether) appear in \lb{EQU:rho-d=}{B}.  
The terms in \eqref{EQU:Del_L-last} cancel with \eqref{EQU:ps-on-the-right} in the case where it is applied to the last tensor factor before $[\bi_{p_\ell},\dots,\bi_{q_1}]$.
The terms in \eqref{EQU:Del_R-last} appear in \eqref{EQU:qs-on-the-left} after 
rotating the last tensor factor to the beginning.
\end{itemize}

Thus, all terms can be identified and cancel, which completes the proof of the fact that $\rho$ is a chain map.

Finally, we show that $\rho$ models the constant loops map. By Theorem \ref{necklacesandloopspace}, there is an isomorphism 
\begin{eqnarray}
\varphi \colon H_\bullet (\mathbb{L}_\bullet(X)) \xrightarrow{\cong} H_\bullet(L|X|),
\end{eqnarray}
where $H_\bullet(L|X|)$ denotes the singular homology of $L|X|$. (See also \cite{RS} where a similar isomorphism is constructed via the cellular chains on a cell complex having the homotopy type of $L|X|$.) 

We argue that
\begin{eqnarray}
   \varphi \circ H_\bullet(\rho)=H_\bullet(i) \colon H_\bullet(X) \to H_\bullet(L|X|),
\end{eqnarray}
where $i$ is the chain map defined as the composition \[i \colon C_\bullet(X) \xrightarrow{f} S_\bullet(|X|) \xrightarrow{S_\bullet(\iota)} S_\bullet(L|X|),\] for $f$ the natural map from simplicial chains on $X$ to singular chains on $|X|$ and $\iota$ the embedding of constant loops. This follows from the Acyclic Models Theorem (\cite[Theorem Ib]{EM}). In fact, on the set of standard simplices simplices $\mathcal{M}= \{ \Delta^0, \Delta^1, \Delta^2, \ldots \}$, the simplicial chains functor $C_\bullet$ is free, the functor $\mathbb{L}_\bullet$ is acyclic, and the natural maps $\varphi\circ H_0(\rho)$ and $H_0(i)$ coincide. The Acyclic Models Theorem implies that there is a natural extension, unique up to natural chain homotopy, of $\varphi\circ H_0(\rho)=H_0(i)$ to the chain level.  

\section{An alternate model for the constant loops map}

We now describe in detail the structure used in the definition of $\chi \colon C_\bullet(X) \to \mathbb{L}_\bullet(X)$ in \eqref{EQU:chi} and prove Theorem \ref{Theorem2}. This provides a conceptual description for the sub-complex of constant loops inside $\mathbb{L}_\bullet(X)$ in terms of higher structure associated to the underlying chains $C_\bullet(X)$.

Recall that, assuming $R$ is a field, one may associate to any pointed topological space $(Z,b)$ a graded cocommutative Hopf $R$-algebra $H_\bullet(\Omega_bZ;R)$ given by the $R$-homology of the space of (Moore) loops in $Z$ based at $b$ with product induced by concatenation of loops and coproduct induced by the diagonal map and the K\"unneth theorem. The antipode is induced by the map sending a loop to its inverse. A base point free version of this construction associates to any topological space $Z$ a category enriched in graded cocommutative $R$-coalgebras. This is constructed by first considering $\mathcal{P}Z$, the topologically enriched category of (Moore) paths in $Z$ with concatenation as composition and then applying the $R$-homology functor on each morphism space $\mathcal{P}Z(a,b)$ to obtain a category $\mathbb{H}_\bullet(\mathcal{P}Z)$ enriched over graded cocommutative $R$-coalgebras. 

In what follows, for any simplicial complex $X$, we describe a chain-level lift to $(\mathbb{P}_\bullet(X),d)$ of the structure described above on $H_\bullet(\mathcal{P}|X|)$. When working at the chain level, we may allow coefficients in an arbitrary commutative ring $R$, since we never apply the K\"unneth theorem. More precisely, in \ref{SEC:nabla_0} and  \ref{SEC:S}, we describe an enrichment of $(\mathbb{P}_\bullet(X),d,\mu)$ in the category of dg coassociative coalgebras with the property of having an antipode. Then, in \ref{SEC:nabla_1}, we describe a compatible chain homotopy for the cocommutativity of the coalgebra structure of this enrichment.

The formulas in Sections \ref{SEC:nabla_0} and \ref{SEC:nabla_1} are drawn from \cite{B}. The main idea behind finding these formulas is that each chain complex $\mathbb{P}_\bullet(X)(a,b)$ may be identified with the chain complex of normalized chains on a cubical set, which has a natural diagonal chain approximation (originally defined by Serre) for which a compatible chain homotopy for cocommutativity can be given explicitly. In a sense, Sections \ref{SEC:nabla_0},  \ref{SEC:S}, and \ref{SEC:nabla_1} below are recollections of structures already constructed in the literature, but adapted to our setting to prove Theorem \ref{Theorem2} in Section \ref{proof1.2}.

\subsection{The coproduct $\nabla_0 \colon \mathbb{P}_\bullet(X) \to \mathbb{P}_\bullet(X) \otimes \mathbb{P}_\bullet(X)$}\label{SEC:nabla_0}

We describe a dg coassociative coalgebra structure on $\mathbb{P}_\bullet(X)$ generalizing a construction from \cite{B}, where it is described in the setting of $1$-reduced simplicial sets. An more general construction in the setting of arbitrary simplicial sets is discussed in \cite[Section 4.1]{R}.

\begin{definition}
Given any $a,b \in X_0$ denote by \[\mathbb{P}_\bullet(X)(a,b)= R\langle\{a\}\rangle \square \big( \bigoplus_{p=0}^{\infty} C_\bullet(X)^{\square p} \big) \square R\langle \{b\} \rangle. \] 
In other words, $\mathbb{P}_\bullet(X)(a,b) \subseteq \mathbb{P}_\bullet(X)$ is generated by all those $(\sigma_1 | \cdots |\sigma_p)$ such that $\mathsf{s}(\sigma_1)=a$ and $\mathsf{t}(\sigma_p)=b$. Note that $d \colon \mathbb{P}_\bullet(X) \to \mathbb{P}_{\bullet-1}(X)$ restricts to a map 
\[d\colon \mathbb{P}_\bullet(X)(a,b) \to \mathbb{P}_{\bullet-1}(X) (a,b),\]
so each $(\mathbb{P}_\bullet(X)(a,b),d)$ is a sub-complex of $(\mathbb{P}_\bullet(X),d)$. Define a coproduct 
\begin{equation}
\nabla_0 \colon \mathbb{P}_\bullet(X)(a,b) \to \mathbb{P}_\bullet(X)(a,b) \otimes \mathbb{P}_\bullet(X)(a,b)
\end{equation}
on any $(\sigma) \in \mathbb{P}_{k-1}(X)(a,b)$ (for simplicity denoted by $\sigma=[0, 1, \ldots, k] \in X_k$, where we assume $a=0$ and $b=k$) by
\begin{multline*} 
\nabla_0(\sigma)= 
    \underset{ \{\bi=(\bi_1,\ldots,\bi_l )  | 0<\bi_1<\ldots < \bi_l <k \} }{\sum}(-1)^{\epsilon(\bi)}  
    \\ \cdot\Big( [0, \ldots, \bi_1]\Big| [\bi_1, \ldots, \bi_2]\Big| \cdots \Big| [\bi_{l}, \dots, k]\Big)  \otimes \Big([0,\bi_1, \bi_2, \dots, \bi_l,k]\Big),
\end{multline*}
where $\epsilon(\bi)= \sum_{j=1}^{l+1}(j-1)(\bi_j-\bi_{j-1})$ interpreting $\bi_0=0$ and $\bi_{l+1}=k$ when $I=(\bi_1, \ldots, \bi_l)$. In particular, when $\sigma=[a,b]$, then $\nabla_0([a,b])= \big([a,b]\big) \otimes \big([a,b]\big)$.
Note that the second factor in $\nabla_0(\sigma)$ is always a sequence of length $1$.

The map $\nabla_0$ is extended to generators $(\sigma_1 | \cdots | \sigma_p\big)$ with $p>1$ as a monoid map. The counit
\begin{eqnarray}\label{counit}
\varepsilon \colon \mathbb{P}_{\bullet}(X)(a,b) \to R
\end{eqnarray}
is determined by $\varepsilon(\sigma)=1_R$ if $\sigma \in X_1 \cup X_1^{-1}$ and $\varepsilon(\sigma)=0$ if $\sigma \in X_k$ for $k>1$. The coproducts $\nabla_0$ on the $\mathbb{P}_\bullet(X)(a,b)$ induce a coproduct on $\mathbb{P}_\bullet(X)= \bigoplus_{a,b \in X_0} \mathbb{P}_\bullet(X)(a,b)$, which we also denote by $\nabla_0$. 
\end{definition}

\begin{proposition}\label{nabla_0equations}
For each $a,b \in X_0$, $(\mathbb{P}_\bullet(X)(a,b), d, \nabla_0)$ is a dg coassociative  coalgebra with counit $\varepsilon \colon \mathbb{P}_{\bullet}(X)(a,b) \to R$ and these satisfy the bimonoid compatibility:
\begin{eqnarray}
    \nabla_0 \circ\mu= (\mu \otimes \mu) \circ (\text{id} \otimes \tau \otimes \text{id}) \circ (\nabla_0 \otimes \nabla_0)
    \end{eqnarray}
\end{proposition}

For a proof of the above, see the more general discussion in \cite[Section 4]{R} or a special case in \cite[Theorem 2.10]{B}. However, the proposition may be verified by direct calculation. Note that the formula for $\nabla_0$ is the same as formula \cite[2.9(3)]{B} extended to include elements $\sigma \in X_1 \cup X_1^{-1}$, where it is defined as $\nabla_0(\sigma)=\sigma \otimes \sigma$. 

\subsection{The antipode $S \colon \mathbb{P}_\bullet(X) \to \mathbb{P}_\bullet(X)$}\label{SEC:S}

We now show that the bimonoid $(\mathbb{P}_\bullet(X),d, \mu, \nabla_0)$ has the property of having an antipode map. (Just in like the bialgebra setting, if an antipode exists, it is unique.)

\begin{proposition} \label{antipodeequations}
There is a chain map $S \colon \mathbb{P}_\bullet(X) \to \mathbb{P}_\bullet(X)$ satisfying the following antipode equations for any $x \in \mathbb{P}_\bullet(X)(a,b)$

\begin{align} 
\mu \circ (\text{id} \otimes S) \circ \nabla_0 (x) &= \varepsilon(x) \text{id}_a,\\
 \mu \circ (S \otimes \text{id}) \circ \nabla_0(x)
&= \varepsilon(x) \text{id}_b,
\end{align}
and
\begin{eqnarray}
S  \circ \mu = \mu \circ (S \otimes S) \circ \tau
\end{eqnarray}
\end{proposition}
\begin{proof}
We define $S_{k-1}(\sigma) \in \mathbb{P}_{k-1}(X)$ for any $\sigma \in X_k$, $k \geq 1$, by induction on $k$ and then we let $S\big(\sigma_1 | \cdots | \sigma_p\big)= \pm \big(S_{|\sigma_p|-1}(\sigma_p)| \ldots |S_{|\sigma_1|-1}(\sigma_1)\big)$.

If $k=1$ then $\sigma=[a,b]$, for some $a,b \in X_0$, and define $S_0( [a,b]) = [b,a]$. The antipode equations are immediately satisfied for $S_0$. 

Suppose we have defined $S_l$ for all $l<k-1$ and let $\sigma=[0,\ldots,k] \in X_k$. Write 
\begin{align*}
\nabla_0([0,\ldots,k] )=&
(-1)^{1+\cdots +(k-1)} \cdot\big([0,1]|\cdots|[k-1,k]\big) \otimes \big( [0,\ldots,k]\big)
\\ &\quad\quad + \widetilde{\nabla}_0([0,\ldots,k]) + \big([0,\ldots, k]\big) \otimes \big([0,k]\big),
\end{align*}
where $\widetilde{\nabla}_0([0,\ldots,k])$ only involves simplices of dimension less than $k$. We want to define $S_{k-1}([0,\ldots, k])$ so that the antipode equations hold, in particular, we want 
\begin{align*} 
\mu \circ (\text{id} \otimes S) \circ \nabla_0([0,\ldots,k] ) = &
(-1)^{1+\cdots +(k-1)} \big([0,1]|\cdots|[k-1,k]\big)S_{k-1}\big( [0,\ldots,k]\big) 
\\& + \mu \circ (\text{id} \otimes S) (\widetilde{\nabla}_0([0,\ldots,k])) 
 + \big([0,\ldots, k] | [k,0] \big)
 \\  = & 0
\end{align*}
since $\varepsilon([0,\dots,k])=0$ when $k>1$. Solving for $S_{k-1}([0,\ldots, k])$ by taking inverses on both sides, forces us to define
\begin{align*}
    S_{k-1}\big( [0,\ldots,k] \big) = &(-1)^{2+\cdots + (k-1)} \big([k,k-1]|\cdots|[2,1] |[1,0]| \mu \circ (\text{id} \otimes S)(\widetilde{\nabla}_0([0,\ldots,k])) \big)
    \\
    &+ (-1)^{2+\cdots + (k-1)}\big([k,k-1]|\cdots|[2,1] |[1,0]|[0,\ldots, k] | [k,0]\big).
\end{align*} 
All terms on the right hand side are defined by induction. A routine calculation yields that the second antipode equation is also satisfied and that $S$ is a chain map.

\end{proof}

\subsection{The chain homotopy $\nabla_1 \colon \mathbb{P}_\bullet(X) \to \mathbb{P}_\bullet(X) \otimes \mathbb{P}_\bullet(X)$}\label{SEC:nabla_1}

We describe a chain homotopy between $\nabla_0$ and $\nabla_0^{\text{op}}=\tau \circ \nabla_0$ satisfying an appropriate derivation compatibility. 
\begin{definition}
Given any $a,b\in X_0$ we define a degree $+1$ coproduct
\begin{eqnarray}
    \nabla_1 \colon \mathbb{P}_\bullet(X)(a,b) \to \mathbb{P}_\bullet(X)(a,b)  \otimes \mathbb{P}_\bullet(X)(a,b) 
\end{eqnarray}
as follows. On $\mathbb{P}_0(X)(a,b)$, $\nabla_1$ is declared to be zero. For any $\sigma=[0,\ldots,k] \in X_k$ and $k>1$ we define
\begin{align*}
\nabla_1([0,\ldots,k]) = &
\sum_{l=1}^{k-1} \underset{((\bi_1,\ldots, \bi_m),(\bj_1,\ldots, \bj_n)) \in \mathcal{T}(k,l)}{\sum} (-1)^{\epsilon(\bi,\bj)} 
\\ 
& \cdot\Big([0,\bi_1,\ldots,\bi_m, l,l+1,\ldots,\bj_1]\Big|[\bj_1, \ldots, \bj_2]\Big| \cdots \Big|[\bj_n, \ldots, k]\Big)
\\
& \otimes\Big([0,\ldots \bi_1]\Big|[\bi_1, \ldots, \bi_2]\Big| \cdots \Big|[\bi_m, \ldots, l,l+1, \ldots, \bj_1, \bj_2, \ldots, \bj_n, k]\Big),
\end{align*}
where the indexing set is 
\begin{multline*}
    \mathcal{T}(k,l)= \{ (\bi=(\bi_1,\ldots, \bi_m),\bj=(\bj_1,\ldots, \bj_n)) | 
\\
0<\bi_1< \cdots <\bi_m < l \text{ and } l <\bj_1 < \cdots < \bj_n <  k\}
\end{multline*}
and for any $(\bi,\bj) \in \mathcal{T}(k,l)$ the corresponding sign is given by
\[\epsilon(\bi,\bj)= \sum_{r=1}^{m+1} (r-1) (\bi_r - \bi_{r-1}) + \sum_{s=1}^{n+1} (s-1) (\bj_s - \bj_{s-1}) + (m + l)(n+k).\]
In particular, $\nabla_1([a,b])=0$ for any $[a.b] \in X_1 \cup X_1^{-1}$ and $\nabla_1([a,z,b])= \big([a,z,b]\big) \otimes \big([a,z,b]\big)$ for any $[a,z,b] \in X_2$. For any $\big(\sigma_1 | \cdots | \sigma_p\big)$, $p>1$, we extend $\nabla_1$ as
\begin{align*} \nabla_1 \big(\sigma_1 | \cdots | \sigma_p\big) =& 
\sum_{r=1}^p (-1)^{|\sigma_1|+ \cdots + |\sigma_{r-1}| -r+1} 
\\ & \cdot \nabla_0^{\text{op}}(\sigma_1) \cdot \ldots \cdot \nabla_0^{\text{op}}(\sigma_{r-1}) \cdot \nabla_1(\sigma_{r}) \cdot \nabla_0(\sigma_{r+1}) \cdot \ldots \cdot\nabla_0(\sigma_{p}),
\end{align*}
where the $\cdot$ symbol denotes the product \[\mu_{\mathbb{P}_\bullet(X) \otimes \mathbb{P}_\bullet(X)} \colon 
(\mathbb{P}_\bullet(X) \otimes \mathbb{P}_\bullet(X)) \square (\mathbb{P}_\bullet(X) \otimes \mathbb{P}_\bullet(X))  \to \mathbb{P}_\bullet(X) \otimes \mathbb{P}_\bullet(X)\]
induced by $\mu$. 
\end{definition}
The following proposition follows from a straightforward computation, which may also be found in the proof of \cite[Theorem 2.2]{B}.

\begin{proposition} \label{nabla_1equations}
For each $a, b\in X_0$, $\nabla_1$ induces a coproduct  \[\nabla_1 \colon \mathbb{P}_\bullet(X)(a,b)\to \mathbb{P}_\bullet(X)(a,b)\otimes \mathbb{P}_\bullet(X)(a,b)\]
satisfying
\begin{eqnarray} \label{nabla1chainhomotopy}
    (d\otimes \text{id} + \text{id} \otimes d) \circ \nabla_1 + \nabla_1 \circ d = \nabla_0^{\text{op}} - \nabla_0
\end{eqnarray}
and
\begin{eqnarray} \label{nabla_1derivation}
\nabla_1 \circ \mu = \mu_{\mathbb{P}_\bullet(X) \otimes \mathbb{P}_\bullet(X)} \circ (\nabla_0^{\text{op}} \otimes \nabla_1 + \nabla_1 \otimes \nabla_0).
\end{eqnarray}

\end{proposition}

\subsection{Proof of Theorem \ref{Theorem2}} \label{proof1.2}
Throughout this section, we use Sweedler notation as described in \eqref{notationDelta} and \eqref{notationnabla}. 

\begin{definition}
We define the map $\chi \colon (C_\bullet(X), \partial) \to (\mathbb{L}_\bullet(X),D)$ as in \eqref{EQU:chi} by setting $\chi(\sigma)=\big(\text{id}_{\sigma})\sigma$ for $\sigma \in X_0$, and for $\sigma \in X_k$ with $k>0$, we define
\begin{multline*}
\chi(\sigma)= (-1)^{|(\sigma)^{0,1}|} \ \big( S(( \sigma)^{0,1}) \big)   \sigma^{0,2} 
+\big( (\sigma)^{{1,1}} | S( (\sigma ) ^{1,2}) \big)  \mathsf{s}(\sigma)
\\
+ (-1)^{ |\sigma'||\sigma''|+|(\sigma')^{0,1}|} \big(( \sigma'')^{1,1} | S(( \sigma'')^{1,2}) | S(( \sigma' )^{0,1})\big)  \sigma'^{\text{  }0,2}
\end{multline*}
\end{definition}
We will prove that $\chi$ is a chain map. In order to organize this computation, we will write $\chi$ as the composition of three maps 
\begin{equation}
C_\bullet(X) \xrightarrow{\iota} \mathbb{L}^{\text{ad}}_\bullet(X) \xrightarrow{T} \mathbb{L}^{\text{ad-op}}_\bullet(X) \xrightarrow{\psi} \mathbb{L}_\bullet(X)
\end{equation}
and check each of these is a chain map in three propositions below.

\begin{definition}
We start by describing the two new chain complexes  $\mathbb{L}^{\text{ad}}_\bullet(X)$ and $\mathbb{L}^{\text{ad-op}}_\bullet(X)$. Both of these complexes are generated by those elements $\sigma_0 \otimes \big(\sigma_1 | \cdots | \sigma_p\big) \in  C_\bullet(X) \otimes \mathbb{P}_\bullet(X)$
such that $\mathsf{t}(\sigma_0)=\mathsf{s}(\sigma_1)= \mathsf{t}(\sigma_p)$. Note these generators are different than the generators of $\mathbb{L}_{\bullet}(X)$, i.e., these are not necklaces in $X$ as described in Section \ref{introduction}). The differential \[ D^{\text{ad}} \colon \mathbb{L}^{\text{ad}}_\bullet(X) \to \mathbb{L}^{\text{ad}}_{\bullet-1}(X)\]
is given by $D= \widetilde{\partial} \otimes \text{id} + \text{id} \otimes d + \delta_R + \delta_L$, where
\begin{eqnarray} 
\delta_R(\sigma \otimes \alpha) =  (-1)^{|\sigma'| + |(\sigma'')^{0,2}||\alpha|} \sigma' \otimes  \big((\sigma'')^{0,1}|\alpha | S((\sigma'')^{0,2})\big) 
\end{eqnarray}
and
\begin{eqnarray} \delta_L(\sigma \otimes \alpha)= \varepsilon \big(\sigma'\big) \sigma'' \otimes \alpha = \partial_0\sigma \otimes \alpha,
\end{eqnarray}
where $\partial_0\sigma$ denotes the $0$-face of $\sigma$. The differential 
\[ D^{\text{ad-op}} \colon \mathbb{L}^{\text{ad-op}}_\bullet(X) \to \mathbb{L}^{\text{ad-op}}_{\bullet-1}(X)
\]
is given by $D^{\text{ad-op}}= \widetilde{\partial} \otimes \text{id} + \text{id} \otimes d + \delta_R^{\text{op}} + \delta_L$
where we now use the opposite coproduct $\nabla_0^{\text{op}}$ instead of $\nabla_0$ as follows (recall the notation $\nabla_0(\sigma)=(\sigma)^{0,1}\otimes (\sigma)^{0,2}$ from \eqref{notationnabla}):
\begin{equation*} \delta_R^{\text{op}}(\sigma \otimes \alpha)= 
(-1)^{|\sigma'| + |(\sigma'')^{0,2}||\alpha|+ |(\sigma'')^{0,2}||(\sigma'')^{0,1}|} \sigma' \otimes  \big((\sigma'')^{0,2}|\alpha | S((\sigma'')^{0,1})\big)
\end{equation*}
A routine calculation yields that both $D^{\text{ad}} \circ D^{\text{ad}}=0$ and $D^{\text{ad-op}} \circ D^{\text{ad-op}}=0$. Note that $\delta_R$ and $\delta_R^{\text{op}}$ use the adjoint action of $\mathbb{P}_\bullet(X)$ on itself using $\nabla_0$ and $\nabla_0^{\text{op}}$, respectively. 
\end{definition}

\begin{definition}
With the above definition, we now define \[\iota \colon C_{\bullet}(X) \to \mathbb{L}_\bullet^{\text{ad}}(X)\] by
\[ \iota(\sigma)= \sigma \otimes \text{id}_{\mathsf{t}(\sigma)}.\]
\end{definition}
\begin{proposition}\label{propiota} The map $\iota \colon C_{\bullet}(X) \to \mathbb{L}_\bullet^{\text{ad}}(X)$ is a chain map. 
\end{proposition}
\begin{proof}
First observe that, since the counit $\varepsilon$ of $\nabla_0$ is only non-zero on the $1$-simplices of $X$ (which are elements in $\mathbb{P}_0(X)$), we have $\varepsilon(\sigma')\sigma''= \partial_{0}\sigma$ and $\varepsilon(\sigma'') \sigma'= \partial_{|\sigma|}\sigma$, the first and last faces of $\sigma \in X_k$, respectively. We use this fact, together with the antipode equations from Proposition \ref{antipodeequations} to compute:
\begin{align*}
D^{\text{ad}}(\sigma \otimes \text{id}_{\mathsf{t}(\sigma)})&= \widetilde{\partial}\sigma \otimes \text{id}_{\mathsf{t}(\sigma)}  \pm \sigma' \otimes \big(( \sigma'')^{0,1}|S( (\sigma'')^{0,2})\big) + \varepsilon(\sigma') \sigma'' \otimes \text{id}_{\mathsf{t}(\sigma)} 
\\
& =\widetilde{\partial}\sigma \otimes \text{id}_{\mathsf{t}(\sigma)} \pm \varepsilon(\sigma'')\sigma' \otimes \text{id}_{\mathsf{s}(\sigma'')} + \varepsilon(\sigma') \sigma'' \otimes \text{id}_{\mathsf{t}(\sigma)} 
\\
&=\widetilde{\partial}\sigma \otimes \text{id}_{\mathsf{t}(\sigma)} \pm \partial_{|\sigma|}\sigma \otimes \text{id}_{\mathsf{t}(\partial_{|\sigma|}\sigma)} + \partial_0\sigma \otimes \text{id}_{\mathsf{t}(\partial_0\sigma)}
\\ &= \iota(\partial \sigma).
\end{align*}
\end{proof}
\begin{definition}
Next we define 
\[ T \colon \mathbb{L}_\bullet^{\text{ad}}(X) \to \mathbb{L}_\bullet^{\text{ad-op}}(X) \]
by 
\[ T= \text{id} + T_0,\]
where
\[
T_0(\sigma \otimes \alpha) = (-1)^{|(\sigma)^{1,2}||\alpha|} \sigma' \otimes \big( (\sigma'')^{1,1}|\alpha| S((\sigma'')^{1,2})\big).
\]
\end{definition}
\begin{proposition}\label{propT}
The map $T \colon \mathbb{L}_\bullet^{\text{ad}}(X) \to \mathbb{L}_\bullet^{\text{ad-op}}(X)$ is a chain map.
\end{proposition}
\begin{proof}
We use Propositions \ref{nabla_0equations}, \ref{antipodeequations}, \ref{nabla_1equations}, and the coassociativity of $\Delta$ in the following computation. Write $\sigma'\otimes \sigma'' \otimes \sigma'''$ for $(\Delta \otimes \text{id}) \circ \Delta (\sigma)= (\text{id} \otimes \Delta)  \circ \Delta (\sigma)$. First, we compute
\begin{align}
D^{\text{ad-op}}(T(\sigma \otimes \alpha))=&
 \label{term1}
 \widetilde{\partial}\sigma \otimes \alpha 
\\ & \label{term2}
\pm \sigma \otimes d \alpha 
\\ & \label{term3}
\pm \sigma' \otimes \big( (\sigma'')^{0,2}|\alpha|S((\sigma'')^{0,1})\big)
\\ & \label{term4}
+ \varepsilon(\sigma')\sigma'' \otimes \alpha
\\ & \label{term5}
\pm (\widetilde{\partial}\sigma)' \otimes \big( (\widetilde{\partial}\sigma'')^{1,1}|\alpha|S((\widetilde{\partial}\sigma'')^{1,2})\big)
\\ & \label{term6}
\pm \sigma' \otimes \big( (\sigma''|\sigma''')^{1,1}|\alpha|S( (\sigma''|\sigma''')^{1,2})\big)
\\ & \label{term7}
\pm \sigma' \otimes \big( (\sigma'')^{1,1}|d\alpha| S((\sigma'')^{1,2})\big)
\\ & \label{term8}
\pm \varepsilon(\sigma')\sigma'' \otimes \big((\sigma''')^{1,1}|\alpha|S((\sigma''')^{1,2})\big)
\\ & \label{term9}
\pm \sigma' \otimes \big( (\sigma'')^{0,2}|\alpha|S((\sigma'')^{0,1})\big)
\\ & \label{term10}
\pm \sigma' \otimes \big( (\sigma'')^{0,1}|\alpha|S((\sigma'')^{0,2})\big)
\\ & \label{term11}
\pm \sigma' \otimes \big( (\sigma'')^{0,2}|(\sigma''')^{1,1}|\alpha|S( (\sigma''')^{1,2})|S((\sigma'')^{0,1})\big).
\end{align}
In the above computation \eqref{term1}, \eqref{term2}, \eqref{term3} and \eqref{term4} come from applying $D^{\text{ad-op}}$ to the identity component of $T(\sigma \otimes \alpha)$. The terms \eqref{term5}, \eqref{term6}, \eqref{term7}, \eqref{term9}, and \eqref{term10} come from applying $d$ to $\big((\sigma'')^{1,1}|\alpha|S((\sigma '')^{1,2})\big)$ and using the chain homotopy equation \eqref{nabla1chainhomotopy}. Terms \eqref{term8} and \eqref{term11} come from applying $\delta_L$ and $\delta_R^{\text{op}}$ to $T_0(\sigma \otimes \alpha)$, respectively. Note that \eqref{term3} cancels with \eqref{term9}.

We now explain how all the remaining terms above cancel with those below:
\begin{align}
    TD^{\text{ad}}(\sigma \otimes \alpha)= &
     \label{Tterm1}
    T(\widetilde{\partial}\sigma \otimes \alpha)
    \\ & \label{Tterm2}
    \pm T(\sigma \otimes d \alpha)
    \\ & \label{Tterm3}
    + T(\varepsilon(\sigma')\sigma'' \otimes \alpha) 
    \\ & \label{Tterm4}
    \pm T(\sigma' \otimes  \big((\sigma'')^{0,1}|\alpha | S((\sigma'')^{0,2})\big)).
\end{align}
Term \eqref{Tterm1} cancels with \eqref{term1} + \eqref{term5}, \eqref{Tterm2} with \eqref{term2} + \eqref{term7}, and \eqref{Tterm3} with \eqref{term4} + \eqref{term8}.  Finally, \eqref{Tterm4} cancels with \eqref{term10} + \eqref{term6} + \eqref{term11}.
In fact, \eqref{term10} is the identity component in  \eqref{Tterm4} and the fact that \eqref{term6} + \eqref{term11} corresponds to the $T_0$ component of \eqref{Tterm4} follows from the compatibility equation \eqref{nabla_1derivation}. 
\end{proof}
\begin{definition}
Finally, we define 
\[ \psi \colon \mathbb{L}_\bullet^{\text{ad-op}}(X) \to \mathbb{L}_\bullet(X)
\]
by $\psi(\sigma \otimes \alpha)= \big(\alpha\big)\sigma$ if $\sigma \in X_0$ and
\[ \psi(\sigma\otimes \alpha)= (-1)^{|\alpha|(|\sigma|-1+|(\sigma)^{0,1}|)}\big(\alpha|S((\sigma)^{0,1})\big)\sigma^{0,2}\]
for any $\sigma \in X_k$, $k>0$. 
\end{definition}
\begin{proposition}\label{proppsi}
The map  $\psi \colon \mathbb{L}_\bullet^{\text{ad-op}}(X) \to \mathbb{L}_\bullet(X)$ is a chain map. 
\end{proposition}
\begin{proof}
We have
\begin{align} \nonumber
D\psi(\sigma \otimes \alpha)=& \pm D(\big(\alpha|S((\sigma)^{0,1})\big)\sigma^{0,2})
\\ =& \label{Dpsi1}
\pm \big(d\alpha|S((\sigma)^{0,1})\big)\sigma^{0,2}
\\ &\label{Dpsi2}
\pm \big(\alpha|\widetilde{\partial} S((\sigma)^{0,1})\big)\sigma^{0,2}
\\ &\label{Dpsi3}
\pm \big(\alpha| \underline{\widetilde{\Delta}}S((\sigma^{0,1}))\big) \sigma^{0,2}
\\ &\label{Dpsi4}
\pm \big(\alpha|S((\sigma)^{0,1})\big)\widetilde{\partial}\sigma^{0,2}
\\ &\label{Dpsi5}
\pm \big(\alpha|S((\sigma)^{0,1})|\sigma^{0,2 '} \big)\sigma^{0,2 ''}
\\ &\label{Dpsi6}
\pm  \big(\sigma^{0,2 ''}|\alpha|S((\sigma)^{0,1}) \big)\sigma^{0,2 '}. 
\end{align}
Note that terms \eqref{Dpsi2} and \eqref{Dpsi3} above correspond to $\pm \big(\alpha|dS((\sigma)^{0,1})\big)\sigma^{0,2}$ using $d=\widetilde{\partial} + \widetilde{\Delta}$.

We explain how all the terms above cancel with the ones below:
\begin{align}
\psi D^{\text{ad-op}}(\sigma \otimes \alpha) =&
 \label{psiD1}
\psi(\widetilde{\partial}\sigma \otimes \alpha)
\\ & \label{psiD2}
\pm \psi (\sigma \otimes d \alpha)
\\ & \label{psiD3}
+ \varepsilon(\sigma')\psi(\sigma'' \otimes \alpha)
\\ & \label{psiD4}
\pm \psi(\sigma' \otimes \big( \sigma ''^{0,2}|\alpha|S((\sigma '')^{0,1})\big)).
\end{align}
Term \eqref{psiD1} cancels with \eqref{Dpsi2} + \eqref{Dpsi4} and term \eqref{psiD2} cancels with \eqref{Dpsi1}. We may write \eqref{psiD3} as
\begin{align} \label{rewrite1}
\varepsilon(\sigma')\psi(\sigma'' \otimes \alpha)= \big(\alpha|S( (\sigma'|\sigma'')^{0,1})| \sigma{'} ^{0,2}\big)\sigma''^{0,2}
\end{align}
and \eqref{psiD4} as
\begin{align} \label{rewrite2}
    \psi(\sigma' \otimes \big( \sigma^{'' 0,2}|\alpha|S((\sigma '')^{0,1})\big))= \big( \sigma ''^{0,2}|\alpha|S((\sigma'|\sigma'')^{0,1})\big) \sigma{'}^{0,2}.
\end{align}
Finally, we see that \eqref{rewrite1} + \eqref{rewrite2}
cancels with \eqref{Dpsi3} + \eqref{Dpsi5} + \eqref{Dpsi6} using Proposition \ref{nabla_0equations} (in particular, that $\nabla_0$ commutes with $\widetilde{\Delta}$). 
\end{proof}

With the above work, we are now able to prove Theorems \ref{Theorem2} and \ref{Theorem-chain-homotopic}.
\begin{proof}[Proof of Theorems \ref{Theorem2} and \ref{Theorem-chain-homotopic}] An easy computation shows that $\chi = \psi \circ T \circ\iota$. Then Propositions \ref{propiota}, \ref{propT}, \ref{proppsi} imply that $\chi$ is a chain map. 

Unraveling the formula for $\chi$, we observe that for any $\sigma \in X_k$ with $k=0,1$, we have $\rho(\sigma)=\chi(\sigma)$, so $H_0(\rho)=H_0(\chi)$ and $H_1(\rho)=H_1(\chi)$. A similar acyclic models argument as the one used to prove Theorem \ref{Theorem1} yields that $\chi$ and $\rho$ are naturally chain homotopic. 
\end{proof}

\section*{Acknowledgments} 
MR would like to thank Alex Takeda and Florian Naef for fruitful discussions. MR acknowledges support by NSF Grant DMS 2405405.

\end{document}